\documentclass[12pt]{amsart}
\usepackage[letterpaper,margin=1.1in]{geometry}
\usepackage{graphicx}
\usepackage{amssymb}
\usepackage[mathscr]{eucal}
\usepackage{amsmath}
\usepackage{amsthm}
\usepackage{bigints}

\newtheorem{theorem}{Theorem}[section]
\newtheorem{lemma}[theorem]{Lemma}
\newtheorem{corollary}[theorem]{Corollary}
\newtheorem{proposition}[theorem]{Proposition}

\theoremstyle{definition}
\newtheorem{definition}[theorem]{Definition}

\theoremstyle{remark}
\newtheorem{remark}[theorem]{Remark}

\newtheorem*{ack}{Acknowledgements}

\newcommand{\be}{\begin{equation}}
\newcommand{\ee}{\end{equation}}
\newcommand{\ben}{\begin{equation*}}
\newcommand{\een}{\end{equation*}}
\newcommand{\bes}{\begin{eqnarray}}
\newcommand{\ees}{\end{eqnarray}}
\newcommand{\besn}{\begin{eqnarray*}}
\newcommand{\eesn}{\end{eqnarray*}}
\newcommand{\txt}{\textrm}
\newcommand{\mres}{\mathbin{\vrule height 1.6ex depth 0pt width
0.13ex\vrule height 0.13ex depth 0pt width 0.7ex}}

\def\Xint#1{\mathchoice
{\XXint\displaystyle\textstyle{#1}}%
{\XXint\textstyle\scriptstyle{#1}}%
{\XXint\scriptstyle\scriptscriptstyle{#1}}%
{\XXint\scriptscriptstyle\scriptscriptstyle{#1}}%
\!\int}
\def\XXint#1#2#3{{\setbox0=\hbox{$#1{#2#3}{\int}$ }
\vcenter{\hbox{$#2#3$ }}\kern-.57\wd0}}

\def\dashint{\Xint-}

\numberwithin{equation}{section}
\numberwithin{figure}{section}

\begin{document}
\title[Geometry of Rectifiable Sets with Carleson and Poincar\'{e} Conditions]{On the Geometry of Rectifiable Sets with Carleson and Poincar\'{e}-type Conditions}

\author{Jessica Merhej}
\thanks{The author was partially supported by NSF DMS-0856687 and DMS-1361823 grants.}
\address{Department of Mathematics\\ University of Washington\\ Box 354350\\ Seattle, WA 98195}
\email{jem05@uw.edu, j.e.merhej@gmail.com}
\date{October 16, 2015}
\begin{abstract}
A central question in geometric measure theory is whether geometric properties of a set translate into analytical ones. In 1960, E. R. Reifenberg proved that if a closed subset $M$ of $\mathbb{R}^{n+k}$ is well approximated by $n$-planes at every point and at every scale, then $M$ is a locally bi-H\"older image of an $n$-plane. Since then, Reifenberg's theorem has been refined in several ways in order to ensure that $M$ is a bi-Lipschitz image of an $n$-plane. In this paper, we consider an $n$-Ahlfors regular rectifiable subset $M$ of $\mathbb{R}^{n+1}$ that satisfies a  Poincar\'{e}-type inequality. Then, we show that a Carleson-type condition on the oscillation of the unit normal of $M$ is sufficient to prove that $M$ is contained inside a bi-Lipschitz image of an $n$-plane. We also show that the Poincar\'{e}-type inequality encodes geometrical information about $M$; namely it implies that $M$ is quasiconvex. \end{abstract}

\maketitle

\setcounter{tocdepth}{1}
\tableofcontents

\section{Introduction} \label{intro}
The Plateau problem has played a fundamental role in the development of geometric measure theory and geometric analysis. In dimension two, it was solved (independently) by Douglas and Rad\'{o} (see \cite{Ra} and \cite{Do}) in 1930. It took time to make sense of the question in higher dimensions. Reifenberg \cite{Re} approached the question of regularity for solutions to the Plateau problem in 1960. His initial tool was the \emph{topological disk theorem}. In recent years, there has been renewed interest in this result and its proof. Roughly speaking, the topological disk theorem states that if an $n$-dimensional subset $M$ of $\mathbb{R}^{n+k}$ is well approximated by an $n$-plane at every point and at every scale, then locally, $M$ is a bi-H\"older  image of the unit ball in $\mathbb{R}^{n}$. To be more precise, we state the theorem here:

\begin{theorem} \label{TDT} (Topological Disk Theorem) \label{th1} \cite{Re} \cite{DT1}
For all choices of integers $n>0$ and $k>0$,  and $0 < \tau < 10^{-1}$, we can find $\epsilon>0$ such that the following holds: Let $M \subset \mathbb{R}^{n+k}$ be a closed set that contains the origin, and suppose that for $x \in M \cap B_{10}(0)$  and $0 < r \leq 10$ we can find an $n$-dimensional affine subspace $P(x,r)$ of $\mathbb{R}^{n+k}$ that contains $x$ such that 

\begin{equation} \label{eq1} 
\txt{dist}(y, P(x,r)) \leq \epsilon r \quad \txt{for} \,\,y \in M \cap B_{r}(x),
\end{equation} 
and 
\begin{equation} \label{eq2} 
\txt{dist}(y, M) \leq \epsilon r \quad \txt{for} \,\, y \in P(x,r) \cap B_{r}(x).
\end{equation} 

Then, there exists a bijective mapping $g: \mathbb{R}^{n+k} \rightarrow \mathbb{R}^{n+k}$ such that 
\begin{equation} \label{eq3} 
|g(x) - x| \leq \tau \quad \txt{for} \, \,x \in \mathbb{R}^{n+k},
\end{equation}

\begin{equation} \label{eq4} \frac{1}{4} |x-y|^{1+\tau} \leq |g(x) - g(y)| \leq 3 |x-y|^{1-\tau}, \end{equation}
for $x, \, y \in \mathbb{R}^{n+k}$ such that $|x-y| \leq 1$, and if we set $P = P(0,10)$, 
\begin{equation} \label{eq5} M \cap B_{1}(0) = g(P) \cap B_{1}(0).\end{equation}
\end{theorem}

A set satisfying inequalities (\ref{eq1}) and (\ref{eq2}) is said to be an $\epsilon$-\emph{Reifenberg flat} set and the map $g$ constructed in the theorem above is called a \emph{Reifenberg parametrization} of $M$. Semmes \cite{Se1, Se2} uses a Reifenberg-type parametrization to get good parametrizations of chord arc surfaces with small constant. David, De Pauw, and Toro \cite{DDT} give a generalization of Reifenberg's theorem in $\mathbb{R}^{3}$. The works by David \cite{Da1,Da2}, partially generalizing Taylor's \cite{Ta} results rely on the Reifenberg-type parametrization constructed in \cite{DDT}. In \cite{To2}, Toro refines Reifenberg's condition in order to guarantee the existence of better parametrizations, and so do David and Toro in \cite{DT1}. David and Toro \cite{DT2} also use Reifenberg-type parametrization to get snowflake-like embeddings of flat metric spaces, a work related to the results Cheeger and Colding \cite{CC} who use a Reifenberg-type parametrization to parametrize the limits of manifolds with Ricci curvature bounded from below. Colding and Naber improve this latter result in \cite{CN}. Moreover, Naber and Valtorta \cite{NV1, NV2} use a variation of Reifenberg's parametrization to study the regularity of stationary and minimizing harmonic maps.\\

A question which motivated many of the papers mentioned above, is whether the map $g$ in Theorem \ref{TDT} is $K$-bi-Lipschitz, that is whether there exists a constant $K \geq 1$, such that for all $x, \,y \, \in \mathbb{R}^{n+k}$, we have
\begin{equation} \label{bilip} K^{-1} \, |x-y| \leq |g(x) - g(y)|\leq K \, |x-y|. \end{equation}
Notice that in Theorem \ref{TDT}, the smaller $\epsilon$ is, the closer the bi-H\"older  exponent is to 1, that is, the closer the map $g$ is to being bi-Lipschitz. Also, it is known that any Lipschitz domain with sufficiently small Lipschitz constant is Reifenberg flat, for a suitable choice of $\epsilon$ depending on the Lipschitz constant. However, the converse is not true in general. In fact, the Von Koch snowflake (with sufficiently small angle) is an example of a Reifenberg flat set which is not Lipschitz (see \cite{To1}). Finding bi-Lipschitz parametrizations of sets is a central question in areas of geometry and metric analysis. Bi-Lipschitz functions in metric spaces play the role played by diffeomorphisms in smooth manifolds (and Lipschitz functions play the role played by smooth functions). Moreover, many concepts in metric analysis, for instance metric dimensions, are invariant under bi-Lipschitz mappings. Another example where Lipschitz and bi-Lipschitz mappings are of utmost importance is the theory of rectifiability in geometric measure theory. An $n$-dimensional rectifiable subset of $\mathbb{R}^{n+k}$, up to a set of measure zero, is a set contained in a countable union of Lipschitz images of $\mathbb{R}^{n}$. Rectifiable sets are a measure theoretic generalization of smooth surfaces that provide the appropriate setting to study geometric variational problems. For a set to be rectifiable, it does not necessarily have to be smooth, but it inherits some characteristics of smooth surfaces. In particular, rectifiable sets are characterized by having \emph{approximate tangent planes} almost everywhere.  Moreover, if an Ahlfors regular set is a Lipschitz or bi-Lipschitz image of $\mathbb{R}^{n}$ in the ambient space $\mathbb{R}^{n+k}$ for some $k \geq 1$, then the set is uniformly rectifiable, where the latter is a quantitative version of rectifiability.\\

So, it is very interesting to know what conditions guarantee that the map $g$ in Theorem \ref{TDT} is bi-Lipschitz. David and Toro \cite{DT1} give several results, each providing sufficient conditions on the set $M$ so that $g$ is bi-Lipschitz. One of the conditions involves the Jones numbers 

\begin{equation} \label{eq66} \beta_{\infty}(x,r) = \frac{1}{r}  \inf_{P} \big \{ \sup \{\txt{dist}(y,P); \,\, y \in B_{r}(x) \} \big\}, \end{equation}
where $x \in M \cap B_{10}(0)$, $0 < r \leq 10$, and the infimum is taken over all $n$-dimensional affine subspaces $P$ of $ \mathbb{R}^{n+k}$, passing through $x$.\\

It is not surprising that the $\beta_{\infty}$ numbers play a role here. They  were introduced by Jones in the Traveling Salesman Problem \cite{J1}, and then used by Bishop and Jones in \cite{BJ1} and \cite{BJ2}, and by Lerman and many others in the context of Lipschitz or nearly Lipschitz parametrizations (see \cite{DS1,DS2,J2,J3,L,P}).\\

Now, consider the function $ J_{\infty}(x)= \displaystyle \sum_{k\geq 0}\beta^{2}_{\infty}(x,10^{-k})$, where $x \in M \cap B_{10}(0)$. David and Toro prove \cite{DT1} that if a set $M$ is $\epsilon$-Reifenberg flat, and if the function $J_{\infty}$ is uniformly bounded on $M \cap B_{10}(0)$, then $M$ is a bi-Lipschitz image of an $n$-dimensional affine subspace in $\mathbb{R}^{n+k}$. They also prove the same result while considering the possibly smaller \footnote{In the case where $M$ is locally Ahlfors regular with Ahlfors regularity constant $C_{M}$, we have $\beta_{1}(x,r) \leq C_{M} \, \beta_{\infty}(x,r)$.} numbers 
 $\beta_{1}$-numbers 

\begin{equation} \label{eq6} \beta_{1}(x,r) = \inf_{P} \frac{1}{r^n} \int_{M \cap B_{r}(x)} \frac{\txt{dist}(y,P)}{r}\, d \mathcal{H}^{n}(y) ,\end{equation}
where $x \in M \cap B_{10}(0)$, $0< r \leq 10$, $\mathcal{H}^{n}$ is the $n$-dimensional Hausdorff measure, and the infimum this time, is taken over all $n$-dimensional affine subspaces $P$ of $ \mathbb{R}^{n+k}$, passing through $B_{r}(x)$, (and not necessarily through $x$). One can think of the $\beta_{1}$-numbers as a weak version of the $\beta_{\infty}$ numbers. Analogous to the function $J_{\infty}$, consider the function $ J_{1}(x)= \displaystyle \sum_{k\geq 0}\beta^{2}_{1}(x,10^{-k})$, where $x \in M \cap B_{10}(0)$. Then, David and Toro prove 

\begin{theorem} \label{theorem 1.2} \footnote{Note that in this theorem, there is no apriori assumption of Ahlfors regularity.}\label{thTD} (see Theorem 1.4 in \cite{DT1})
Suppose that $n$, $k$, and $M$ are as in Theorem \ref{TDT}. Let $\epsilon >0$ small enough, depending on $n$ an $k$. Assume that for every $x \in M \cap B_{10}(0)$ and for every $0 < r \leq 10$, we can find an $n$-dimensional affine subspace $P(x,r)$ of $\mathbb{R}^{n+k}$ that contains $x$ such that (\ref{eq1}) and (\ref{eq2}) hold. Moreover, suppose there exists a positive number $N$ such that for all $x \in M \cap B_{10}(0)$, we have 
$J_{1}(x) := \displaystyle \sum_{k\geq 0}\beta^{2}_{1}(x,10^{-k}) \leq N$. Then, the mapping $g$ provided by Theorem \ref{TDT} is $K$-bi-Lipschitz, that is, (\ref{bilip}) holds, with the bi-Lipschitz constant $K$ depending only on $n$, $k$, and $N$.
\end{theorem}

It was very interesting to find a condition involving the $\beta_{1}$-numbers sufficient to guarantee a local bi-Lipschitz parametrization of $M$ (from Theorem \ref{thTD}), since a previous result by David and Semmes \cite{DS2} stated that for an $n$-Ahlfors regular subset $M$ of $\mathbb{R}^{n+k}$, a Carleson condition on the $\beta_{1}$-numbers
\begin{equation*} \int_{M \cap B_{r}(x)}\int_{0}^{r} \beta^{2}_{1}(x,r)\frac{dt}{t} \, d \mathcal{H}^{n}(y) \leq C_{0} \, r^{n}, \end{equation*}
where $x \in M$, $0 < r\leq 1$ and $C_{0}$ is a constant that depends only on $n$, $k$, and the Ahlfors regularity constant 
 is a necessary condition for $M$ to be 
(locally) a bi-Lipschitz image of an $n$-plane (see \cite{DT1}, remark 15.6). Carleson-type conditions which are sufficient for $M$ to admit a bi-Lipschitz parametrization have been studied (see \cite{To2}). In \cite{To2}, Toro studies a Carleson-type condition on the Reifenberg flatness (equations (\ref{eq1}) and (\ref{eq2})) which yields a bi-Lipschitz parametrization. As a corollary, she obtains an interesting result for a special type of \emph{chord arc surfaces with small constant}, that is CASSC.

\begin{definition}
Let $M$ be a connected $C^{2}$ hyper-surface in $\mathbb{R}^{n+1}$ such that $M \cup \{ \infty \}$ is a $C^{2}$ hyper-surface in $\mathbb{R}^{n+1} \cup \{ \infty \}$. Let $\nu(x)$ denote a choice of unit normal to $M$. Let $||\nu||_{*}$ denote the BMO norm of $\nu$, that is 
\begin{equation*} ||\nu||_{*} = \sup_{x \in M, r>0}  \frac{1}{\mathcal{H}^{n}(M \cap B_{r}(x))}\int_{M \cap B_{r}(x)}| \nu(y) - \nu_{x,r}| \, d \mathcal{H}^{n}(y),\end{equation*} 
where  $\nu_{x,r} =  \displaystyle \dashint_{B_{r}(x)} \nu(y) \, d \mathcal{H}^{n}(y) = \displaystyle \frac{1}{\mathcal{H}^{n}(M \cap B_{r}(x))}\int_{M \cap B_{r}(x)} \nu(y) \, d \mathcal{H}^{n}(y)$ denotes the average of the unit normal $\nu$ on the ball $B_{r}(x)$.

Suppose that there exists $\gamma >0$ small enough such that $||\nu||_{*} \leq \gamma$ and the following holds
\begin{equation*} |<x-y, \nu_{x,r}>| \leq \gamma \, r \quad \forall\,\, x \in M,\, 0 < r \leq 1 \,\, \txt{and} \,\, y \in M \cap B_{r}(x). \end{equation*}
Then, $M$ is called a \emph{chord arc surface with small constant}. 
\end{definition}

Thus, CASSC are $C^{2}$ hyper-surfaces in $\mathbb{R}^{n+1}$ that have small BMO norm, and at every point $x$ and scale $r$, they are close to the $n$-plane whose normal is $\nu_{x,r}$. These hyper-surfaces were introduced by Semmes \cite{Se1}. He proves that they can be locally parametrized by a $C^{0,\alpha}$ homeomorphism, for any $\alpha < 1$. It is then natural to ask if they admit a local bi-Lipschitz parametrization. In \cite{To2}, Toro proves the following theorem about CASSC:

\begin{theorem} \label{CASSC} (see Corollary 5.1 in \cite{To2})
Suppose $M$ is a CASSC. There exists $\delta >0$ and $\epsilon>0$, depending only on $n$ such that if $||\nu||_{*} \leq \delta$ and 
\begin{equation} \label{CC2} \int_{0}^{\infty} \sup_{x \in M} \left ( \,\,\dashint_{M \cap B_{2r}(x)}| \nu(y) - \nu_{x,2r}|^{p} \right )^{\frac{2}{p}} \frac{dr}{r} \leq \epsilon^{2}, \end{equation}
for some $p>n$, then $M$ admits a local $K$-bi-Lipschitz parametrization, with the bi-Lipschitz constant $K$ depending on $\epsilon$, $\delta$, and the dimension $n$.
\end{theorem}

In this paper, we generalize Theorem \ref{CASSC}. We relax both the regularity condition imposed on the hyper-surface $M$, and the Carleson condition on the oscillation of the unit normal of $M$, and prove the existence of a local bi-Lipschitz parametrization for $M$. So what conditions do we want to start with? We consider rectifiable sets of co-dimension 1, which are Ahlfors regular, and satisfy a Poincar\'{e}-type inequality. As mentioned earlier, rectifiable sets are characterized by having \emph{approximate tangent planes} almost everywhere (see Definitions \ref{rect} and \ref{ats} for precise definitions of rectifiability and approximate tangent planes). Thus, if $M$ is rectifiable of co-dimension 1, it admits (generalized) unit normals almost everywhere. Notice that at every point in $M$ where an approximate tangent plane exists, there are two choices for the direction of the generalized unit normal. Later, we impose a Carleson-type condition that ensures a coherent choice of generalized unit normal for $M$.\\

Let $M$ be an $n$-dimensional rectifiable set in $\mathbb{R}^{n+1}$. Denote by $\mu$ the $n$-Hausdorff measure restricted to $M$, that is,  $\mu =$  \( \mathcal{H}^{n} \mres M\). Suppose that $M$ is $n$-Ahlfors regular with Ahlfors regularity constant $C_{M}$ (see Definition 2.11 for the definition of $n$-Ahlfors regular sets and the Ahlfors regularity constant). We note here that CASSC are in particular $n$-Ahlfors regular, (see \cite{Se1}). The Poincar\'{e}-type inequality we consider on $M$ is the following:\\ 
\\
For all $x \in M$, $r > 0$, and $f$ a locally Lipschitz function on $\mathbb{R}^{n+1}$, we have 
\be \label{eqp} \dashint_{B_{r}(x)} \left| f(y) - f_{x,r} \right| \, d \mu(y) \leq C_{P} \, r \, \left(\,\dashint_{B_{2r}(x)}|\nabla^{M}f(y)|^{2} \, d \mu(y) \right)^{\frac{1}{2}}, \ee
where $C_{P}$ denotes the Poincar\'{e} constant that appears here, $f_{x,r}$ is the average of the function $f$ on $B_{r}(x)$(see (\ref{average}) for precise definition),
and $\nabla^{M}f(y)$ denotes the tangential derivative of $f$ (see (\ref{td}) for the definition the tangential gradient). \\

We remark here that Semmes has proved in \cite{Se3} that the Poincar\'{e}-type inequality (\ref{eqp}) is satisfied by CASSC. In fact, this is the motivation behind our asking that the rectifiable set $M$ satisfies this Poincar\'{e}-type inequality. This inequality is different from the usual Poincar\'{e} inequality on Euclidean space (see \cite{EG} p. 141). For instance, in (\ref{eqp}), the average of the oscillation of $f$ is bounded by its tangential derivative and not the usual derivative; moreover, the ball on the right hand side of (\ref{eqp}) has twice the radius of the ball on the left hand side of (\ref{eqp}), which is not the case in the usual Poincar\'{e} inequality. However, (\ref{eqp}) fits perfectly with the Poincar\'{e} inequality that Riemannian manifolds with Ricci curvature bounded from below satisfy (see \cite{HK} p.46) once we take the metric $g$ to be the pullback of the Euclidean metric to the manifold. Semmes' proof that CASSC satisfy (\ref{eqp}) strongly depends on the fact that the surface is chord arc, and in particular, smooth. In this paper, we assume this inequality, and prove, in the last section, that not all rectifiable sets satisfy (\ref{eqp}). In fact, we prove that this Poincar\'{e}-type inequality (\ref{eqp}) gives connectivity information about $M$.\\

 We are ready to state the main result of this paper:\\

\begin{theorem} \label{MTT'}
Let $M \subset B_{1}(0)$ be an $n$-Ahlfors regular rectifiable set containing the origin, and let $\mu =$  \( \mathcal{H}^{n} \mres M\) be the Hausdorff measure restricted to $M$. Assume that $M$ satisfies the Poincar\'{e}-type inequality (\ref{eqp}). There exists $\epsilon_{0} = \epsilon_{0}(n, C_{M}, C_{P})>0$, such that if there exists a choice of unit normal $\nu$ to $M$ so that

\be \label{103} \int_{0}^{1} \left(\,\dashint_{B_{r}(x)} |\nu(y) - \nu_{x,r}|^{2} \, d \mu \right) \frac{dr}{r} < \epsilon_{0}^{2} \quad \txt{for} \,\, x \in M \cap B_{\frac{1}{10^{3}}}(0), \ee 
then there exists a bijective $K$-bi-Lipschitz map $g: \mathbb{R}^{n+1} \rightarrow \mathbb{R}^{n+1}$ where the bi-Lipschitz constant $K = K(n, C_{M}, C_{P})$, and an $n$-dimensional plane $\Sigma_{0}$, with the following properties:
\be \label{aa} g(z)= z \quad \txt{when} \,\,\, d(z, \Sigma_{0}) \geq 2, \ee
and
\be  \label{bb} |g(z)-z| \leq C_{0} \epsilon_{0} \quad \txt{for} \,\,\, z \in \mathbb{R}^{n+1}, \ee
where $C_{0} = C_{0}(n, C_{M}, C_{P})$. Moreover, 
\be \label{cc} g(\Sigma_{0})\,\, \txt{is a} \,\,\, C_{0} \epsilon_{0} \txt{-Reifenberg flat set}, \ee and 
\be M \cap B_{\frac{1}{10^{3}}}(0) \subset g(\Sigma_{0}). \ee
\end{theorem}

It is worth mentioning here that the theorem states that $M$ \emph{is contained in} a bi-Lipschitz image of an $n$-plane instead of $M$ being exactly a bi-Lipschitz image of an $n$-plane, as proved in Theorems \ref{TDT}, \ref{thTD}, and \ref{CASSC}. However, this is very much expected, since when we drop the assumption of Reifenberg flatness on $M$, we have to deal with the fact that $M$ might be full of holes. \\

 The paper is structured as follows: in section \ref{Pre} we record several definitions and preliminaries. In section \ref{LAD}, we prove a couple of linear algebra lemmas needed to prove Theorem \ref{MTT'}. In section \ref{MBLI}, we prove Theorem \ref{MTT'}. This is done in several steps. First, we define the $\alpha$-numbers

\ben   \alpha(x,r) := \left(\,\dashint_{B_{r}(x)} |\nu(y) - \nu_{x,r}|^{2} \, d \mu \right)^{\frac{1}{2}}, \een
where $x \in M$, and $0 < r \leq 1$.\\
\\
These $\alpha$-numbers play the role that $\beta_{1}$-numbers played for Theorem \ref{thTD} (see Theorem \ref{t1}) \footnote{It is worth mentioning here that Theorem \ref{t1} is the place in which the Poincar\'e inequality is used, where it is applied only to a specific smooth function on $\mathbb{R}^{n+1}$. However, it is the heart of Theorem \ref{t1}, which in turn, is an integral step to proving Theorem \ref{MTT'}.}. Then we prove Theorem \ref{MTT'} using the $\alpha$-numbers, while handling the issue that $M$ might be have many holes. We finish this section by a remarking that the Carleson-type condition (\ref{103}) can be replaced by a more general condition on the $\alpha$ numbers which shows up naturally in the proof of Theorem \ref{MTT'} (see Lemma \ref{l1}, Remark \ref{cargen}, and Theorem \ref{MTT'again}), and Theorem \ref{MTT'} would still hold. \\

In section \ref{QCM}, we show that the Poincar\'{e}-type inequality satisfied by $M$ is interesting by itself, as it encodes some geometric information about $M$. In fact, we show that if a rectifiable set $M$ satisfies (\ref{eqp}), then $M$ is quasiconvex. A set is quasiconvex if any two points in the set are connected by a rectifiable curve, contained in the set, whose length is comparable to the distance between the two points. We finish this section by remarking that the Poincar\'e inequality (\ref{eqp}) is indeed equivalent to the other usual Poincar\'e-type inequalities found in literature that imply quasiconvexity (see \cite{Ch}, \cite{DJS}, \cite{K1} \cite{KM}) . Thus, Theorem \ref{MTT'} still holds if one replaces (\ref{eqp}) with any of those Poincar\'e-type inequalities.

\section{Preliminaries} \label{Pre}
Throughout this paper, our ambient space is $\mathbb{R}^{n+1}$. $B_{r}(x)$ denotes the open ball center $x$ and radius $r$ in $\mathbb{R}^{n+1}$, while $\bar{B}_{r}(x)$ denotes the closed ball center $x$ and radius $r$ in $\mathbb{R}^{n+1}$. $d(.,.)$ denotes the distance function from a point to a set. $\mathcal{H}^{n}$ is the $n$-Hausdorff measure. Finally, constants may vary from line to line, and the parameters they depend on will always be specified in a bracket. For example, $C(n,C_{M})$ will be a constant that depends on $n$ and the Ahlfors regularity constant, that may vary from line to line. \\

We begin by recalling the definition of a Lipschitz and a bi-Lipschitz function:

\begin{definition}
Let $M \subset \mathbb{R}^{n+1}$. A function $ f: M \rightarrow \mathbb{R}$ is called \emph{Lipschitz} if there exists a constant $K>0$, such that for all $x, \, y \in M$ we have
\begin{equation} \label{lip1}
|f(x) - f(y)| \leq K \, |x-y|.
\end{equation}
The smallest such constant is called the \emph{Lipschitz constant} and is denoted by $LIPf$.
\end{definition}

\begin{definition}
A function $ f: \mathbb{R}^{n+1} \rightarrow \mathbb{R}^{n+1}$ is called $K$-\emph{bi-Lipschitz} if there exists a constant $K>0$, such that for all $x, \, y \in \mathbb{R}^{n+1}$ we have
\begin{center}
$K^{-1} |x-y| \leq |f(x) - f(y)| \leq K \, |x-y|.$ 
\end{center}
\end{definition}

Next, we introduce the class of \emph{n-rectifiable} sets:
\begin{definition} \label{rect}
Let $M \subset \mathbb{R}^{n+1}$ be an $\mathcal{H}^{n}$-measurable set. $M$ is said to be countably \emph{n-rectifiable} if 
 \begin{center}$ M \subset M_{o} \cup \left(\displaystyle \bigcup_{i=1}^{\infty}f_{i}(A_{i})\right)$, \end{center}
where $ \mathcal{H}^{n}(M_{o}) = 0$, and $ f_{i} : A_{i} \rightarrow \mathbb{R}^{n+1}$ is Lipschitz, and $A_{i} \subset \mathbb{R}^{n}$, for  $i = 1, 2, \ldots$ 
\end{definition}

 $n$-rectifiable sets are characterized in terms of approximate tangent spaces which we now define:

\begin{definition} \label{ats}
If $M$ is an $\mathcal{H}^{n}$-measurable subset of $\mathbb{R}^{n+1}$. We say that the $n$-dimensional subspace $P(x)$ is the \emph{approximate tangent space of $M$ at $x$}, if
\begin{equation}   \lim_{\lambda \to 0}  \lambda^{-n} \int_{M} {f \left (\lambda^{-1}(y-x)\right) } \, d \mathcal{H}^{n}(y) = \int_{P(x)} f(y) \, d\mathcal{H}^{n}(y) \quad \forall f \in C^1_c(\mathbb{R}^{n+1}, \mathbb{R}).
\end{equation} 
\end{definition}

\begin{remark}
Notice that if it exists, $P(x)$ is unique. From now on, we shall denote the tangent space of $M$ at $x$ by $T_{x}M$.\end{remark}

The following theorem gives the important characterization of $n$-rectifiable sets in terms of approximate tangent spaces:\\
\begin{theorem} \label{atp} (see \cite{Si}; Theorem 11.6 \footnote{See the proof of the \emph{only if} part p. 62, to realize that Theorem \ref{atp} of this paper is a special case of Theorem 11.6 in \cite{Si}.})\\
Suppose $M$ is an $\mathcal{H}^{n}$-measurable subset of $\mathbb{R}^{n+1}$. Then $M$ is countably $n$-rectifiable \emph{if and only if} the approximate tangent space $T_{x}M$ exists for $\mathcal{H}^{n}$-a.e. $ x \in M$.
\end{theorem}

\begin{remark} \label{RN}
Notice that Theorem \ref{atp} and the fact that $M$ is co-dimension 1 guarantee the existence of a generalized unit normal to $M$ for $\mathcal{H}^{n}$-a.e. $ x \in M$. However, at each of these points, there are two choices for the direction of the generalized unit normal. 
\end{remark}

We also need to define the notion \emph{Reifenberg flatness}: 
\begin{definition} Let $M$ be an $n$-dimensional subset of $\mathbb{R}^{n+1}$. We say that $M$ is $\epsilon$-Reifenberg flat for some $\epsilon >0$, if for every $x \in M$   and $0 < r \leq \frac{1}{10^{4}}$, we can find an $n$-dimensional affine subspace $P(x,r)$ of $\mathbb{R}^{n+1}$ that contains $x$ such that 

\begin{equation*}
d(y, P(x,r)) \leq \epsilon r \quad \txt{for} \,\,y \in M \cap B_{r}(x),
\end{equation*} 
and 
\begin{equation*} 
d(y, M) \leq \epsilon r \quad \txt{for} \,\, y \in P(x,r) \cap B_{r}(x).
\end{equation*} 
 \end{definition}

\begin{remark}
Notice that the above definition is only interesting if $\epsilon$ is small, since any set is 1-Reifenberg flat. \end{remark}

\noindent In the proof of our theorems, we need to measure the distance between two $n$-dimensional planes. We do so in terms of normalized local Hausdorff distance:

\begin{definition}
Let $x$ be a point in $\mathbb{R}^{n+1}$ and let $r >0$. Consider two closed sets $E,\,F \subset \mathbb{R}^{n+1}$ such that both sets meet the ball $B_{r}(x)$. Then,
\begin{equation*} d_{x,r}(E,F) = \frac{1}{r} \, \txt{Max} \left\{ \sup_{y \in E \cap B_{r}(x)} \txt{dist}(y,F) \,\,; \sup_{y \in F \cap B_{r}(x)} \txt{dist}(y,E) \right\} \end{equation*}
is called the normalized Hausdorff distance between $E$ and $F$ in $B_{r}(x)$.
\end{definition}

Finally, we recall the definition of an $n$-Ahlfors regular measure and an $n$-Ahlfors regular set:

\begin{definition}
Let $M \subset \mathbb{R}^{n+1}$ be a closed, $\mathcal{H}^{n}$ measurable set, and let $\mu=$  \( \mathcal{H}^{n} \mres M\) be the $n$-Hausdorff measure restricted to $M$. We say that $\mu$ is $n$-Ahlfors regular if there exista a constant $C_{M} \geq 1$, such that for every $x \in M$ and $ 0< r < 1$, we have 
\begin{equation} \label{alfh}  C_{M}^{-1} \, r^{n} \leq \mu(B_{r}(x)) \leq C_{M} \, r^{n}.\end{equation}

In such a case, the set $M$ is called an $n$-Ahlfors regular set, and $C_{M}$ is referred to as the Ahlfors regularity constant.
\end{definition}

\section{Linear Algebra Digression} \label{LAD}
To prove our main theorem, we need the following two linear algebra lemmas. Since they are independent results, let us digress a little bit and prove them here. \\

\textbf{Notation:}\\
Let $V$ be an affine subspace of $\mathbb{R}^{n+1}$ of dimension $k$, $ k \in \{ 0, \dots, n-1\}$. 
Denote by $N_{\delta}(V)$, the $\delta$-neighbourhood of $V$, that is, 
\ben N_{\delta}(V) = \left\{ x \in \mathbb{R}^{n+1} \,\,\txt{such that} \,\, d(x,V) < \delta \right\}. \een

\begin{lemma}
\label{l1} Let $M$ be an $n$-Ahlfors regular subset of $\mathbb{R}^{n+1}$, and let $\mu =$  \( \mathcal{H}^{n} \mres M\) be the Hausdorff measure restricted to $M$. There exists a constant $c_{0} = c_{0}(n, C_{M}) \leq \displaystyle \frac{1}{2}$ such that the following is true: Fix $x_{0} \in M$, $r_{0} < 1$ and let $r = c_{0} \, r_{0}$. Then, for every $V$, an affine subspace of $\mathbb{R}^{n+1}$ of dimension $0 \leq k \leq n-1$,
there exists $x \in M \cap B_{r_{0}}(x_{0})$ such that $x \notin  N_{11 r}(V)$ and $B_{r}(x) \subset B_{2 r_{0}}(x_{0})$. 
\end{lemma}

\begin{proof}
Fix $x_{0} \in M$, $r_{0} < 1$, and $k \in \left\{ 0, \ldots, n-1 \right\}$. Let $V$ be an affine $k$-dimensional subspace of $\mathbb{R}^{n+1}$. Consider $N_{11 r}(V)$, where $r < r_{0}$ is to be determined later. The set \\ $ \mathcal{A} := \left\{ B_{ \frac{r}{5}}(x), \,\, x \in M \cap N_{11 r}(V) \cap B_{r_{0}}(x_{0}) \right\}$ forms a cover for $M \cap N_{11 r}(V) \cap B_{r_{0}}(x_{0})$, and thus by Vitali's theorem, there exists a finite disjoint subset of $\mathcal{A}$, say  $ \mathcal{A^{'}} := \left\{ B_{ \frac{r}{5}}(x_{i}) \right\}_{i=1}^{N}$, such that 
\be \label{24} M \cap N_{11 r}(V) \cap B_{r_{0}}(x_{0}) \subset \bigcup_{i=1}^{N} B_{r}(x_{i}). \ee

Let us start by getting an upper bound for the number of balls $N$, needed to cover  $M \cap N_{11 r}(V) \cap B_{r_{0}}(x_{0})$. Notice that
\be \label{25} \bigcup_{i=1}^{N} B_{\frac{r}{5}}(x_{i}) \subset B^{k}_{r+r_{0}}(a) \times B^{n+1-k}_{12r}(a), \ee
where $a = \pi_{V}(x_{0})$, the orthogonal projection of $x_{0}$ on $V$, $ B^{k}_{r+r_{0}}(a) = V \cap B_{r+r_{0}}(a)$, and $ B^{n+1-k}_{12r}(a) = V^{\perp} \cap  B_{12r}(a)$ where $V^{\perp}$ is the affine subspace, perpendicular to $V$ and passing through $a$. \\
\\
In fact, take $x \in \displaystyle \bigcup_{i=1}^{N} B_{r}(x_{i})$. Then there exists $x_{i} \in M \cap B_{r_{0}}(x_{0}) \cap N_{11r}(V)$, with $i \in \{ 1, \ldots, N\}$ such that $|x - x_{i}| \leq \frac{r}{5}$. Now, write $x$ as $x = (\pi_{V}(x),\pi_{V^{\perp}}(x))$. On one hand, we have  

\bes \label{sm} |\pi_{V}(x) - a| &=& |\pi_{V}(x) - \pi_{V}(x_{0})| \nonumber \\
&\leq&  |\pi_{V}(x) - \pi_{V}(x_{i})| +  |\pi_{V}(x_{i}) - \pi_{V}(x_{0})| \nonumber \\
&\leq& |x - x_{i}| + |x_{i} - x_{0}| \leq r + r_{0},\ees
where in the last step we used the facts that $x_{i} \in B_{r_{0}}(x_{0})$ and $|x - x_{i}| \leq \frac{r}{5}$.\\

On the other hand,
\bes \label{sm'} |\pi_{V^{\perp}}(x) - a| &\leq&  |\pi_{V^{\perp}}(x) - \pi_{V^{\perp}}(x_{i})| +  | \pi_{V^{\perp}}(x_{i}) -  a| \nonumber \\
&\leq& |x - x_{i}| + 11r \leq 12r,\ees
where in the step before the last we used the fact that $x_{i} \in N_{11r}(V)$, and in the last step we used that $|x - x_{i}| \leq \frac{r}{5}$.\\
\\
Combining (\ref{sm}) and (\ref{sm'}), we get (\ref{25}).\\
\\
Since the balls in $\mathcal{A^{'}}$ are disjoint, then by taking the Lebesgue measure on each side of (\ref{25}), we get 
\bes N \omega_{n+1} \left(\frac{r}{5} \right)^{n+1} &\leq& \omega_{k}\,(r_{0} + r)^{k}\, \omega_{n+1-k} \,({12 r})^{n+1-k} \nonumber \\
&\leq& C(n,k) \,(r_{0} + r)^{k} \, {r}^{n+1-k} \nonumber \\
&\leq&  C(n,k) \, r_{0}^{k} \, {r}^{n+1-k} 
\ees
 where in the last step, we used the fact that $r<r_{0}$. 
Thus,
\be \label{26} N \leq C(n,k) \, r_{0}^{k} \, r^{-k}. \ee

Now, we want to use the fact that $\mu$ is Ahlfors regular to compare the $\mu$-measures of the sets $N_{11 r}(V) \cap  B_{r_{0}}(x_{0})$ and $ B_{r_{0}}(x_{0})$.\\

On one hand, since $\mu$ is lower Ahlfors regular and $x_{0} \in M$, we have by (\ref{alfh})
\be \label{30} \mu\big(B_{r_{0}}(x_{0})\big) \geq C_{M}^{-1} \,r_{0}^{n}. \ee

On the other hand, by (\ref{24}), the fact that $\mu$ is upper Ahlfors regular and $x_{i} \in M$ for all $i \in \{ 1, \ldots, N \}$, and by (\ref{26}), we get

\bes \label{27} \mu\big(N_{11 r}(V) \cap  B_{r_{0}}(x_{0})\big) &=&  \mu\big(M \cap N_{11 r}(V) \cap  B_{r_{0}}(x_{0})\big) \nonumber \\  &\leq& \sum_{i=1}^{N} \mu\big(B_{r}(x_{i})\big) \nonumber \\ &\leq& C_{M} \, N\, r^{n} \nonumber\\
&\leq& C(n,k, C_{M}) \,r_{0}^{k} \, r^{n-k},  \ees
Let us denote by $C_{1}$ the constant $C(n,k, C_{M})$ we get from (\ref{27}). From now till the end of the proof, $C_{1}$ will stand for exactly this constant. Hence, (\ref{27}) becomes 

\be  \label{28} \mu\big(N_{11 r}(V) \cap  B_{r_{0}}(x_{0})\big) \leq C_{1} \,r_{0}^{k} \, r^{n-k}. \ee

Thus, if we pick $r$ such that 
\be \label{29} r^{n-k} < \frac{C_{M}^{-1}}{C_{1}} \, r_{0}^{n-k}, \ee
then
\be \label{31} C_{1} \,r_{0}^{k} \, r^{n-k} < C_{M}^{-1} \, r_{0}^{n} .\ee

Comparing (\ref{31}) with (\ref{30}) and (\ref{28}), we get 
\ben  \mu\big(N_{11r}(V) \cap  B_{r_{0}}(x_{0})\big) <  \mu\big(B_{r_{0}}(x_{0})\big), \een

and thus, there exists a point $x \in M \cap  B_{r_{0}}(x_{0})$ such that $ x \notin  N_{11 r}(V)$.\\
\\
Notice that the proof of the lemma would have been done if the statement allowed for $r = c(n,k, C_{M}) r_{0}$ \big(see(\ref{29})\big). In fact, we have shown that for every $k \in \left\{ 0, \ldots, n-1 \right\}$, and for every $V$,  an affine $k$-dimesional subspace of $\mathbb{R}^{n+1}$, there is a constant $c(n,k, C_{M})$ such that if $r \leq c(n,k, C_{M}) r_{0}$, then we can find a point $x \in M \cap  B_{r_{0}}(x_{0})$ such that $ x \notin  N_{11 r}(V)$.\\
\\
Now, take $r = c_{0}\, r_{0}$ where $c_{0} < \txt{min}\{ c(n,0, C_{M}), \ldots c(n,n-1, C_{M})\}$. First, notice that $c_{0}$ is a constant depending only on $n$ and $C_{M}$. Moreover, when $V$ is an affine $k$-dimensional subspace of $\mathbb{R}^{n+1}$, $k \in \left\{ 0, \ldots, n-1 \right\}$, we have $ r = c_{0}\, r_{0} \leq  C(n,k, C_{M}) r_{0}$. Thus, there exists a point $x \in M \cap  B_{r_{0}}(x_{0})$ such that $ x \notin  N_{11 r}(V)$.\\
Without loss of generality, we can assume that $c_{0} \leq \displaystyle \frac{1}{2}$.  The fact that $B_{r}(x) \subset B_{2 r_{0}}(x_{0})$ follows directly from the fact that $r < r_{0}$, and the proof is done.
\end{proof}

\begin{remark}
Let us note here that as stated in the lemma above, the dimension of the affine subspace $V$ is allowed to be 0. In fact, if $V$ is a single point, say $V = \left\{ y_{0} \right\}$, then $N_{\delta}(V)  = B_{\delta}(y_{0})$, and the proof follows exactly as above.\end{remark}

Moreover, the dimension  $k$ of $V$ has $n-1$ as an upper bound. This is because the lemma fails for $k=n$ (take $M = V= \mathbb{R}^{n}$ and let $x_{0} = 0$).

\begin{lemma} \label{l3}
Fix $R>0$, and let $\{u_{1}, \ldots u_{n} \}$ be $n$ vectors in $\mathbb{R}^{n+1}$. Suppose there exists a constant $K_{0} >0$ such that 

\be \label{44'} |u_{j}| \leq K_{0} \, R \quad \forall j \in \{1,\ldots, n\}. \ee
Moreover, suppose there exists a constant $0 < k_{0} < K_{0}$, such that 
\be \label{44}|u_{1}|\geq k_{0} \, R,\ee
and 
\be \label{45} u_{j} \notin N_{k_{0}R}\big(span\{u_{1}, \ldots u_{j-1}\} \big)  \quad \forall j \in \{2,\ldots, n\}. \ee

Then, for every vector $v \in V:= span\{u_{1}, \ldots u_{n}\} $, $v$ can be written uniquely as 
\be v = \sum_{j=1}^{n} \beta_{j}u_{j},\ee
where \be \label{46} |\beta_{j}| \,\leq K_{1}\frac{1}{R} \, |v|, \quad \forall j \in \{1,\ldots, n\} \ee
with $K_{1}$ being a constant depending only on $n$, $k_{0}$, and $K_{0}$.
\end{lemma}

\begin{proof}
Since the vectors $\{u_{1}, \ldots u_{n} \}$ are linearly independent (by (\ref{45})), then by the Gram-Schmidt process, we construct $n$ orthonormal vectors, $\{e_{1}, \ldots e_{n} \}$ such that 
\be \label{47} span\{u_{1} , \ldots u_{j} \} =  span\{e_{1} , \ldots e_{j} \} \quad \forall j \in \{1,\ldots, n\} ,  \ee
and 
\be \label{48} u_{j} = \sum_{i=1}^{j} \lambda_{j}^{i} \, e_{i} \quad \forall j \in \{1,\ldots, n\}.  \ee

Let us first consider $j=1$. By (\ref{48}), (\ref{44'}), (\ref{44}), and the fact that $e_{1}$ is a unit vector, we have
\be \label{49} u_{1} = \lambda_{1}^{1} e_{1} \quad \txt{with} \,\,\, k_{0} \, R \leq\ |\lambda_{1}^{1}| \leq K_{0}\, R. \ee
For $i=2$, (\ref{48}), (\ref{45}), and (\ref{47}) tell us that
\ben u_{2} = \lambda_{2}^{1} e_{1} + \lambda_{2}^{2} e_{2}, \een
with
\ben u_{2} \notin N_{k_{0}R}\big( span\{u_{1} \}\big) =  N_{k_{0}R}\big( span\{e_{1} \}\big). \een
This means that 
\be \label{137} |\lambda_{2}^{2}| = d\big( u_{2},span\{e_{1}\} \big) \geq k_{0} \, R. \ee
Moreover, from (\ref{44'}) and the fact that the $\{ e_{1}, e_{2}\}$ is a set of orthonormal vectors, we have
\ben |u_{2}| = \sqrt{ (\lambda_{2}^{1})^{2}  + (\lambda_{2}^{2})^{2} } \leq K_{0} \, R ,\een
that is
\ben | \lambda_{2}^{i}| \leq K_{0} \, R \,\,\,\,\,\, i \in \{1,2 \}. \een
Continuing in a similar manner, we get for every $ j \in \{1, \ldots, n\}$,
\be \label{50} |\lambda_{j}^{j}| =  d\big( u_{j},span\{e_{1}, \ldots , e_{j-1}\} \big) \geq k_{0}\, R, \ee
and
\be \label{50'}  | \lambda_{j}^{i}| \leq K_{0} \, R \,\,\,\, \forall i \in \{1, \ldots j \}. \ee

Let $A$ be the $n \times n$ matrix whose $j$-th column is $u_{j}$ written in the orthonormal basis $\{e_{1}, \ldots , e_{n} \}$. Notice that by construction, $A$ is an upper triangular matrix, whose $ij$-th entry is $\lambda_{j}^{i}$, for every $ i \leq j$. \\
\\
Moreover, $A$ is invertible (since all its diagonal entries are non-zero by (\ref{50})), and is the change of basis matrix from the basis $\{u_{1}, \ldots , u_{n} \}$ to the  basis $\{e_{1}, \ldots , e_{n} \}$.\\
\\ 
Now, consider a vector $v \in  V:= span\{u_{1}, \ldots u_{n}\} =  span\{e_{1}, \ldots e_{n}\} $.  Denoting by $v_{u}$ and $v_{e}$ the representation of the vector $v$ in the bases  $\{u_{1}, \ldots , u_{n} \}$ and $\{e_{1}, \ldots , e_{n} \}$ respectively, let us set
\be \label{53} v = \sum_{j=1}^{n} \beta_{j}u_{j} = \sum_{j=1}^{n} \alpha_{j}e_{j}\ee

We know that $v_{e} = A \cdot v_{u}$, that is
\be \label{51} v_{u} = A^{-1} \cdot v_{e}.\ee
Substituting (\ref{53}) in equality (\ref{51}), we get 
\be \label{52} (\beta_{1}, \ldots, \beta_{n}) =  A^{-1} \cdot (\alpha_{1}, \ldots, \alpha_{n}).  \ee

Let us recall here that 
\be \label{m} A^{-1} = \frac{1}{det(A)} adj(A),\ee

where $adj(A)$ is the adjoint matrix of $A$.\\

Now, if we denote by $(\txt{row})_{l}$, the $l$-th row of $adj(A)$, $l \in \{1 \dots n \}$, then by (\ref{50'}) and unravelling the definition of $adj(A)$, we get
\be \label{141} |(\txt{row})_{l}| \leq \sqrt{n} \, K_{0}^{n-1} \,(n-1)! \, R^{n-1} \quad \forall l \in \{ 1 \ldots n \}.  \ee

Moreover, since  $A$ is an upper triangular matrix, whose $j$-th diagonal entry is $\lambda_{j}^{j}$, then by (\ref{50})
\be \label{mm} det(A) = \lambda_{1}^{1} \ldots \lambda_{n}^{n} \geq  k_{0}^{n} \, R^{n}.\ee

We are now ready to get an upper bound on the $\beta_{j}$'s:\\

From (\ref{52}) and (\ref{m}), we can see that for every $j \in \{1, \ldots, n \}$

\be \label{153} \beta_{j} =\frac{1}{det(A)} \, (\txt{row})_{j} \, \cdot (\alpha_{1}, \ldots, \alpha_{n}). \ee 

Thus, by (\ref{153}),  (\ref{mm}), (\ref{141}), (\ref{53}),  and the fact that $\{ e_{1}, \ldots, e_{n} \}$ are an orthonormal set of vectors, we get
\bes \label{154} | \beta_{j}| &\leq& \frac{1}{k_{0}^{n}\, R^{n}} \,  |(\txt{row})_{j}| \, |  (\alpha_{1}, \ldots, \alpha_{n})| \nonumber \\
&\leq&  \frac{1}{k_{0}^{n}\, R^{n}} \,  \sqrt{n} \, K_{0}^{n-1} \,(n-1)! \, R^{n-1} \, |v| \nonumber \\
&=& K_{1} \frac{1}{R} \, |v|,\ees
where $K_{1}$ is a constant depending on $n$, $k_{0}$, and $K_{0}$.
This completes the proof of the lemma.

\end{proof}

\section{M is contained in a bi-Lipschitz image of an $n$-plane} \label{MBLI}

Throughout the rest of the paper, $M$ denotes an $n$-Ahlfors regular rectifiable subset of $\mathbb{R}^{n+1}$ and $\mu =$  \( \mathcal{H}^{n} \mres M\) denotes the Hausdorff measure restricted to $M$. The average of a function $f$ on the ball $B_{r}(x)$ is denoted by 
\begin{equation} \label{average} f_{x,r}= \dashint_{B_{r}(x)} f  \, d \mu(y) =  \displaystyle \frac{1}{\mu(M \cap B_{r}(x))}   \int_{B_{r}(x)} f \, d \mu(y). \end{equation}

Finally, for a locally Lipschitz function $f$ on $\mathbb{R}^{n+1}$, $\nabla^{M}f(y)$ denotes the tangential derivative of $f$ at the point $y \in M$. More precisely, 

\begin{equation} \label{td'} \nabla^{M}f(y) = \nabla (f|_{L}) (y)\end{equation}
where $L := y + T_{y}M$, $f|_{L}$ is the restriction of $f$ on the affine subspace $L$, and $\nabla(f|_{L})$ is the usual gradient of $f|_{L}$.\\

In the special case when $f$ is a smooth function on $\mathbb{R}^{n+1}$, we have
\begin{equation} \label{td} \nabla^{M}f(y) = \pi_{T_{y}M} (\nabla f (y)),\end{equation}
where $\pi_{T_{y}M}$ is the orthogonal projection of $\mathbb{R}^{n+1}$ on $T_{y}M$, and $\nabla f$ is the usual gradient of $f$.\\

In this section, we prove Theorem \ref{MTT'}, the main theorem of this paper. Recall that Theorem \ref{MTT'} states that if there is a choice of unit normal to $M$ such that the Carleson-type condition (\ref{103}) on the oscillation of the unit normal to $M$ is satisfied, and if $M$ satisfies the Poincar\'e-type condition (\ref{eqp}), then $M$ lives inside a bi-Lipschitz image of an $n$-dimensional plane.\\

Let us highlight the main steps needed to prove this theorem.  First, we define what we call the $\alpha$-numbers
\be \label{a'} \alpha(x,r) := \left(\,\dashint_{B_{r}(x)} |\nu(y) - \nu_{x,r}|^{2} \, d \mu \right)^{\frac{1}{2}}, \ee
where $x \in M$, and $0 < r \leq \displaystyle \frac{1}{10}$.\\
These numbers are the most important ingredient to proving our theorem. In Lemma \ref{l5}, we show that  the Carleson condition (\ref{103}) implies that these numbers are small. Moreover, for every point $x \in M$, and series $\displaystyle \sum_{i=1}^{\infty} \alpha^2(x,  10^{-j})$ is finite. Then, in Theorem \ref{t1}, we show that the Poincar\'e-type inequality allows us to construct an $n$-plane $P_{x,r}$ at every point $x \in M$ and every scale $0<r \leq\frac{1}{20}$ where the distance (in integral form) from $M \cap B_{r}(x)$ to $P_{x,r}$ is bounded by $\alpha(x,2r)$. This means, by Lemma \ref{l5}, that those distances are small, and for a fixed point $x$, when we add these distances at the scales $ 10^{-j}$ for $j \in \mathbb{N}$, this series is finite \footnote{ Theorem \ref{t1} implies that the series $\displaystyle \sum_{i=1}^{\infty} \beta_{1}^2(x,  10^{-j})$ is finite. The possible existence of holes in $M$ is the main obstacle, at this point, that does not allow the direct application of Theorem \ref{theorem 1.2}.}. Theorem \ref{t1} is the key point that allows us to use the bi-Lipschitz parametrization that G. David and T. Toro construct in \cite{DT1}. In fact, what they do is construct approximating $n$-planes, and prove that at any two points that are close together, the two planes associated to these points at the same scale, or at two consecutive scales are close in the Hausdorff distance sense. From there, they construct a bi-H\"older  parametrization for $M$. Then, they show that the sum of these distances at scales $ 10^{-j}$ for $j \in \mathbb{N}$ is finite (uniformly for every $x \in M$). This is what is needed for their parametrization to be bi-Lipschitz (see Theorem \ref{t2} below and the definition before it). Thus, the rest of this section is devoted to using Theorem \ref{t1} in order to prove the compatibility conditions between the approximating planes mentioned above, while handling the issue that our set $M$ might be full of holes. \\

Let us begin with the two lemmas that explore the Carleson condition (\ref{103}).

\begin{lemma} \label{l5}
Let $M \subset B_{1}(0)$ be an $n$-Ahlfors regular rectifiable set containing the origin, and let $\mu =$  \( \mathcal{H}^{n} \mres M\) be the Hausdorff measure restricted to $M$. Let $\epsilon > 0$, and suppose that there is a choice of unit normal $\nu$ to $M$ such that \be \label{0} \int_{0}^{1} \left(\,\dashint_{B_{r}(x)} |\nu(y) - \nu_{x,r}|^{2} \, d \mu \right) \frac{dr}{r} < \epsilon^{2}, \quad \forall x \in M. \ee
Then, for every $x \in M$, we have 
\be \label{a} \sum_{j=1}^{\infty} \alpha^{2}(x, 10^{-j}) \leq C \, \epsilon^{2} ,\ee where the $\alpha$-numbers are as defined in (\ref{a'}) and $C = C(n, C_{M})$ is a constant that depends only on $n$ and $C_{M}$.\\
Moreover, for every $x \in M$ and $0 < r \leq \displaystyle \frac{1}{10}$, we have
\be \label{a''} \alpha(x,r) \leq C \, \epsilon ,\ee
 where $C = C(n, C_{M})$. \end{lemma}

\begin{proof}
Let $\epsilon > 0$ and $\nu$ be as described above. Fix $x \in M$. For all $a \in \mathbb{R}^{n+1}$, and for all $0 < r_{0} \leq 1$, we have
\be \label{10} \dashint_{B_{r_{0}}(x)} |\nu(y) - \nu_{x,r_{0}}|^{2} \, d \mu \leq  \,\dashint_{B_{r_{0}}(x)} |\nu(y) - a|^{2} \, d \mu,\ee
since the average $\nu_{x,r_{0}}$ of $\nu$ in the ball $B_{r_{0}}(x)$ minimizes the integrand on the right hand side of (\ref{10}).\\

To prove (\ref{a}), we start by showing
 \be \label{17} \sum_{j=1}^{\infty}  \dashint_{B_{ 10^{-j}}(x)} |\nu(y) -  \nu_{x, 10^{-j}}|^{2} \, d \mu  \leq C(n, C_{M}) \, \sum_{j=0}^{\infty} \int_{ 10^{-j-1}}^{ 10^{-j}} \dashint_{B_{r}(x)} |\nu(y) - \nu_{x,r}|^{2} \, d \mu \, \frac{dr}{r}. \ee

Fix $j \in \mathbb{N}$ and let $r$ be such that 
\be \label{113}  10^{-j-1} < r \leq  10^{-j}, \quad \txt{that is} \quad \frac{1}{ 10^{-j}} \leq \frac{1}{r} < \frac{1}{ 10^{-j-1}}. \ee  

Using (\ref{10}) for $a = \nu_{x,r}$ and $r_{0} =  10^{-j-1}$, (\ref{113}), and the fact that $\mu$ is Ahlfors regular, we get
\bes \label{20} \dashint_{B_{ 10^{-j-1}}(x)} |\nu(y) - \nu_{x, 10^{-j-1}}|^{2} \, d \mu &\leq&  \dashint_{B_{ 10^{-j-1}}(x)} |\nu(y) - \nu_{x,r}|^{2} \, d \mu \nonumber \\
&\leq& C(n, C_{M})\, \dashint_{B_{r}(x)} |\nu(y) - \nu_{x,r}|^{2} \, d \mu , \ees

Dividing both sides of (\ref{20}) by $r$ and then integrating from $ 10^{-j-1}$ to $ 10^{-j}$, we get

\be \label{21} \int^{ 10^{-j}}_{ 10^{-j-1}} \dashint_{B_{ 10^{-j-1}}(x)} |\nu(y) - \nu_{x, 10^{-j-1}}|^{2} \, d \mu \, \frac{dr}{r}\leq C(n, C_{M}) \, \int^{ 10^{-j}}_{ 10^{-j-1}} \dashint_{B_{r}(x)} |\nu(y) - \nu_{x,r}|^{2} \, d \mu  \, \frac{dr}{r} .\ee

Using (\ref{113}) on the left hand side of (\ref{21}) gives us
 \ben \frac{1}{ 10^{-j}} \int^{ 10^{-j}}_{ 10^{-j-1}} dr \dashint_{B_{ 10^{-j-1}}(x)} |\nu(y) - \nu_{x, 10^{-j-1}}|^{2} \, d \mu \leq C(n, C_{M}) \int^{ 10^{-j}}_{ 10^{-j-1}} \dashint_{B_{r}(x)} |\nu(y) - \nu_{x,r}|^{2} \, d \mu  \, \frac{dr}{r},\een
and thus
\be \label{22} \dashint_{B_{ 10^{-j-1}}(x)} |\nu(y) - \nu_{x, 10^{-j-1}}|^{2} \, d \mu \leq C(n, C_{M}) \, \int^{ 10^{-j}}_{ 10^{-j-1}} \dashint_{B_{r}(x)} |\nu(y) - \nu_{x,r}|^{2} \, d \mu  \, \frac{dr}{r}.\ee

Taking the sum over $j$ from $0$ to $\infty$ on both sides of (\ref{22}), we get

\be \label{23} \sum_{j=0}^{\infty} \dashint_{B_{ 10^{-j-1}}(x)} |\nu(y) - \nu_{x, 10^{-j-1}}|^{2} \, d \mu \leq C(n, C_{M}) \sum_{j=0}^{\infty} \int^{ 10^{-j}}_{ 10^{-j-1}} \dashint_{B_{r}(x)} |\nu(y) - \nu_{x,r}|^{2} \, d \mu  \, \frac{dr}{r}, \ee

that is,

\be \label{23'}  \sum_{j=1}^{\infty} \dashint_{B_{ 10^{-j}}(x)} |\nu(y) - \nu_{x, 10^{-j}}|^{2} \, d \mu \leq C(n, C_{M}) \,\sum_{j=0}^{\infty} \int^{ 10^{-j}}_{ 10^{-j-1}} \dashint_{B_{r}(x)} |\nu(y) - \nu_{x,r}|^{2} \, d \mu  \, \frac{dr}{r} \ee
hence finishing the proof of (\ref{17}).\\ 

But, it is trivial to check that 
\be \label{baa} \sum_{j=0}^{\infty} \int_{ 10^{-j-1}}^{ 10^{-j}} \dashint_{B_{r}(x)} |\nu(y) - \nu_{x,r}|^{2} \, d \mu \, \frac{dr}{r} = \int_{0}^{1} \left(\,\dashint_{B_{r}(x)} |\nu(y) - \nu_{x,r}|^{2} \, d \mu \right) \frac{dr}{r}.\ee

Thus, plugging (\ref{baa}), (\ref{a'}), and (\ref{0}) in (\ref{23'}), we get 

\ben \sum_{j=1}^{\infty} \alpha^{2}(x, 10^{-j}) \leq C(n, C_{M}) \, \epsilon^{2} ,\een which is exactly (\ref{a}).

To prove inequality (\ref{a''}), fix $x \in M$ and $0 < r \leq \displaystyle \frac{1}{10}$. Then, there exists $j \geq 1$ such that \be \label{111}  10^{-j-1} < r \leq  10^{-j}, \quad \txt{that is} \quad \frac{1}{ 10^{-j}} \leq \frac{1}{r} < \frac{1}{ 10^{-j-1}}. \ee

Now, using inequality (\ref{10}) for $a = \nu_{x, 10^{-j}}$ and $r_{0} = r$, (\ref{111}), and the fact that $\mu$ is Ahlfors regular, we get (by the same steps used to get (\ref{20})) that

\ben \dashint_{B_{r}(x)} |\nu(y) - \nu_{x,r}|^{2} \, d \mu \leq C(n, C_{M}) \, \dashint_{B_{ 10^{-j}}(x)} |\nu(y) -  \nu_{x, 10^{-j}}|^{2} \, d \mu, \een 

that is, (by (\ref{a'})),

\be \label{baaaa} \alpha^{2}(x,r) \leq C(n, C_{M}) \,  \alpha^{2}(x,  10^{-j}).\ee

Taking the square root on both sides of (\ref{baaaa}) and using (\ref{a}) finishes the proof of (\ref{a''})
\end{proof}

\vspace{0.4cm}
In the next lemma, we use Lemma \ref{l5} to prove that there is a uniform lower bound on $|\nu_{x,r}|$ for all $x \in M$ and $0 < r \leq \frac{1}{10}$.

\begin{lemma} \label{l4}
Let $M \subset B_{1}(0)$ be an $n$-Ahlfors regular rectifiable set containing the origin, and let $\mu =$  \( \mathcal{H}^{n} \mres M\) be the Hausdorff measure restricted to $M$. There exists $\epsilon_{1} = \epsilon_{1}(n, C_{M}) >0$ such that for every $0 < \epsilon \leq \epsilon_{1}$,  if there is a choice of unit normal $\nu$ to $M$ such that \be \label{b} \int_{0}^{1} \left(\,\dashint_{B_{r}(x)} |\nu(y) - \nu_{x,r}|^{2} \, d \mu \right) \frac{dr}{r} < \epsilon^{2}, \,\,\,\,\, \forall x \in M, \ee
then, for all $x \in M$ and $0< r \leq\displaystyle \frac{1}{10}$, 
\be \label{124} |\nu_{x,r}| \geq \frac{1}{2}.\ee
\end{lemma}

\begin{proof}
Fix $x \in M$ and $0 <r \leq \displaystyle \frac{1}{20}$. Let $0 < \epsilon \leq \epsilon_{1}$ (with $\epsilon_{1}$ to be be determined later), and suppose there is a choice of unit normal $\nu$ to $M$ such that (\ref{b}) holds. Define
\ben \nu^{*}_{x,r}(y) = \sup_{\rho \in (0,r)} \, \dashint_{B_{\rho}(y)} | \nu(z) - \nu_{x,2r}| \, d \mu(z), \,\,\,\,\,\,\,\, y \in M. \een
Let $M\left(|\nu - \nu_{x,2r}| \chi_{B_{2r}(x)}\right)$ be the Hardy-Littlewood maximal function (see \cite{CW}, p. 624, for the extension of the  Hardy-Littlewood maximal function to spaces of homogeneous type). By definition,
\ben M \left(|\nu - \nu_{x,2r}| \chi_{B_{2r}(x)}\right)(y) = \sup_{\rho > 0} \,\dashint_{B_{\rho}(y)} | \nu(z) - \nu_{x,2r}|  \chi_{B_{2r}(x)}(z) \, d \mu(z), \,\,\,\,\,\,\,\,\, y \in M. \een
Notice that
\be \label{7} \nu^{*}_{x,r}(y) \leq  M\left(|\nu - \nu_{x,2r}| \chi_{B_{2r}(x)}\right)(y) \,\,\,\,\,\,\,\,\,\,\, \forall y \in M \cap B_{r}(x). \ee
Using (\ref{7}), the fact that $||M(f)||_{L^2} \leq C(C_{M}) ||f||_{L^2}$ (see \cite{CW}, p. 624, 625), Ahlfors regularity of $\mu$, and Lemma \ref{l5}, we have
\bes \label{120}
\left( \, \dashint_{B_{r}(x)} |\nu^{*}_{x,r}|^{2} \, d \mu(y) \right)^{\frac{1}{2}} &\leq& \left(\, \dashint_{B_{r}(x)} | M\left(|\nu - \nu_{x,2r}| \chi_{B_{2r}(x)}\right)(y)|^{2} \, d \mu(y) \right)^{\frac{1}{2}} \nonumber\\
&=& \frac{1}{\mu(B_{r}(x))^{\frac{1}{2}}} \left(\int_{B_{r}(x)} | M\left(|\nu - \nu_{x,2r}| \chi_{B_{2r}(x)}\right)(y)|^{2} \, d \mu(y) \right)^{\frac{1}{2}} \nonumber\\ 
&\leq& \frac{1}{\mu(B_{r}(x))^{\frac{1}{2}}} \left(\int | M\left(|\nu - \nu_{x,2r}| \chi_{B_{2r}(x)}\right)(y)|^{2} \, d \mu(y) \right)^{\frac{1}{2}} \nonumber\\ 
&\leq&  \frac{C(C_{M})}{\mu(B_{r}(x))^{\frac{1}{2}}} \left(\int \left(|\nu - \nu_{x,2r}| \chi_{B_{2r}(x)})(y)\right)^{2} \, d \mu(y) \right)^{\frac{1}{2}} \nonumber\\
&=&  \frac{C(C_{M})}{\mu(B_{r}(x))^{\frac{1}{2}}} \left(\int_{B_{2r}(x)} |\nu(y) - \nu_{x,2r}|^{2} \, d \mu(y) \right)^{\frac{1}{2}}\nonumber\\
&=& C(C_{M}) \, \left(\,\dashint_{B_{2r}(x)} |\nu(y) - \nu_{x,2r}|^{2} \, d \mu(y) \right)^{\frac{1}{2}}\nonumber\\
&=& C(C_{M}) \, \alpha(x,2r)\leq  C(n, C_{M}) \, \epsilon := C_{1} \, \epsilon. \ees

From (\ref{120}), the fact that $M$ is a rectifiable set (so the normal exists $\mu$-a.e. in $M$), and the fact that $\mu$-a.e. point is a Lebesgue point of the function $f(z) =| \nu(z) - \nu_{x,2r}|$ with respect to $\mu$, it is easy to check that there must exist a point $y_{0} \in M \cap B_{r}(x)$, such that $y_{0}$ is a density point for $f(z)$ with respect to $\mu$, $\nu(y_{0})$ exists, and $\nu^{*}_{x,r}(y_{0}) \leq C_{1} \, \epsilon$.\\
\\
So, by definition of $\nu^{*}_{x,r}(y_{0})$, we get
\ben  \sup_{\rho \in (0,r)}\, \dashint_{B_{\rho}(y_{0})} | \nu(z) - \nu_{x,2r}| \, d \mu(z) \leq C_{1} \, \epsilon, \een
that is,
\be \label{8}  \dashint_{B_{\rho}(y_{0})} | \nu(z) - \nu_{x,2r}| \, d \mu(z) \leq C_{1} \, \epsilon, \,\,\,\,\,\,\,\,\,\,\,\, \forall \rho < r. \ee
Taking the limit as $\rho$ approaches $0$, and using the facts that by construction, $y_{0}$ is a density point of  $f(z) =| \nu(z) - \nu_{x,2r}|$ with respect to $\mu$, and $\nu(y_{0})$ exists,  we get
\be \label{9}|\nu(y_{0}) - \nu_{x,2r}| = \lim _{\rho \to 0} \,\,  \dashint_{B_{\rho}(y_{0})} | \nu(z) - \nu_{x,2r}| \, d \mu(z) \leq C_{1} \, \epsilon. \ee
But $\nu(y_{0})$ is a unit vector, and thus by (\ref{9}) and remembering that $\epsilon \leq \epsilon_{1}$, we get
\be \label{b'} \big| |\nu_{x,2r}| - 1 \big| =  \big| |\nu_{x,2r}| -|\nu(y_{0})|  \big| \leq  \big| \nu_{x,2r} -\nu(y_{0})  \big| \leq C_{1} \, \epsilon  \leq C_{1} \, \epsilon_{1}. \ee

Choosing $\epsilon_{1}$ small enough (such that 
$ C_{1} \, \epsilon_{1} \leq \frac{1}{2}$), (\ref{b'}) becomes
$ \big| |\nu_{x,2r}| - 1 \big| \leq \frac{1}{2},$
that is, $ |\nu_{x,2r}| \geq \displaystyle \frac{1}{2}.$ Since $x \in M$ and $0 < r \leq \displaystyle \frac{1}{20}$ were arbitrary, the proof is done.
\end{proof}

As we mentioned before, the construction of the bi-Lipschitz map relies heavily on finding good approximating $n$-planes to $M$. By that, we mean that for a point $x \in M$, and a scale $0<r<\frac{1}{20}$, we would like to find an $n$-plane $P(x,r)$ (not necessarily passing through $x$) such that $M \cap B_{r}(x)$ is close to $P(x,r)$. In the following theorem, with the help of Lemma \ref{l4} and the Poincar\'e-type inequality, we construct a plane $P_{x,r}$ that turn out to be, up to a small translation (as we see later), the plane $P(x,r)$ that we aim to get. \\

\begin{theorem}
\label{t1} Let $M \subset B_{1}(0)$ be an $n$-Ahlfors regular rectifiable set containing the origin, and let $\mu =$  \( \mathcal{H}^{n} \mres M\) be the Hausdorff measure restricted to $M$. Assume that $M$ satisfies the Poincar\'{e}-type inequality (\ref{eqp}).
Let $\epsilon_{1}$ be as in Lemma \ref{l4}, and let $0 < \epsilon \leq \epsilon_{1}$. Suppose there exists a choice of unit normal $\nu$ to $M$ such that 
 \ben \int_{0}^{1} \left(\,\dashint_{B_{r}(x)} |\nu(y) - \nu_{x,r}|^{2} \, d \mu \right) \frac{dr}{r} < \epsilon^{2}, \quad \forall x \in M.\een
Then, for every $x \in M$ and $0< r \leq \displaystyle \frac{1}{20}$, there exists an affine $n$-dimensional plane $P_{x,r}$, whose normal is $\nu_{x,2r}$, and such that 
\be \label{121} \dashint_{B_{r}(x)} \frac{d(y,P_{x,r})}{r} \, d \mu(y) \leq 4 C_{P} \, \alpha(x,2r),\ee
\end{theorem}

\begin{proof}
Fix $x \in M$ and $r \leq \displaystyle \frac{1}{20}$. Let $\epsilon_{1}$, $\epsilon$, and $\nu$ be as above. Consider the function $f$ on $\mathbb{R}^{n+1}$ defined by 
\ben f(y) = \left<y, \nu_{x,2r}\right>, \,\,\,\,\,\,  y \in \mathbb{R}^{n+1}.\een
Notice that $f$ is a smooth function on $\mathbb{R}^{n+1}$, and for every point $y \in M$ where the unit normal $\nu(y)$ exists, (which is almost everywhere in $M$ by Theorem \ref{atp} and the fact that $M$ is rectifiable), we have \be \label{1}|\nabla^{M}f(y)| \leq 2 \, | \nu_{x,2r} - \nu(y)|. \ee In fact,
\ben \nabla^{M}f(y) = \nabla f(y) - \left<\nabla f(y) , \nu(y) \right> \nu(y). \een
But $\nabla f(y) = \nu_{x,2r}$, so
 \besn |\nabla^{M}f(y)| &=& |\nu_{x,2r} - \left<\nu_{x,2r} , \nu(y) \right> \nu(y)| \\
& = & |\nu_{x,2r} - \nu(y) - \left<\nu_{x,2r} - \nu(y), \nu(y) \right> \nu(y)| \\
& \leq& 2 \, | \nu_{x,2r} - \nu(y)|, \eesn
where in the last two steps, we used the fact that $\nu(y)$ is a unit vector.\\

Now, applying the Poincar\'{e} inequality on the function $f$ and the ball $B_{r}(x)$, and using (\ref{1}), we get 

\be \label{2} \frac{1}{r}  \,\dashint_{B_{r}(x)}  \left| \left<y,\nu_{x,2r}\right> - \,\dashint_{B_{r}(x)} \left<z ,\nu_{x,2r}\right> \, d \mu(z)\right| d \mu(y)  \leq 2 C_{P} \left( \,\dashint_{B_{2r}(x)} \left|\nu_{x,2r} - \nu(y)\right|^{2} d \mu(y)\right)^{\frac{1}{2}}.\ee

But $\nu_{x, 2r}$ is a constant vector, so (\ref{2}) can be rewritten as
\be \label{3} \frac{1}{r}  \,\dashint_{B_{r}(x)}  \left| \left<y,\nu_{x,2r}\right> - \left< \,\dashint_{B_{r}(x)} z \, d \mu(z),\nu_{x,2r}\right>\right| d \mu(y) \leq 2 C_{P} \left( \,\dashint_{B_{2r}(x)} \left|\nu_{x,2r} - \nu(y)\right|^{2} d \mu(y) \right)^{\frac{1}{2}} \ee

that is,
\be \label{4} \frac{1}{r} \, \dashint_{B_{r}(x)} \left|\left<y - \, \dashint_{B_{r}(x)} z \, d\mu(z) ,\nu_{x,2r}\right> \right| d \mu(y)  \leq 2 C_{P} \left( \,\dashint_{B_{2r}(x)} \left|\nu_{x,2r} - \nu(y)\right|^{2} d \mu(y) \right)^{\frac{1}{2}}. \ee

We are now ready to choose our plane $P_{{x,r}}$. Let us notice first, that since $x \in M$ and $2r \leq \displaystyle \frac{1}{10}$, (\ref{124}) in Lemma \ref{l4} says that 
\be \label{123} |\nu_{x,2r}|  \geq \frac{1}{2}.\ee

Now, take $P_{x,r}$ to be the plane passing through the point $ c_{x,r} := \,\dashint_{B_{r}(x)} z \, d \mu(z) $, the centre of mass of $\mu$ in the ball $B_{r}(x)$, and whose normal is $\nu_{x,2r}$ (which is possible by (\ref{123})). \\
\\
Then, using (\ref{123}), we have that for every $y \in B_{r}(x)$
\bes \label{5} d(y, P_{x,r}) &=& \left| \left< y - c_{x,r} , \frac{\nu_{x,2r}}{|\nu_{x,2r}|}\right> \right| \nonumber \\
& \leq & 2 \, \left|  \left< y - c_{x,r} , \nu_{x,2r} \right>\right| \nonumber \\
&=& 2 \, \left|  \left< y - \,\dashint_{B_{r}(x)} z \, d \mu(z) , \nu_{x,2r} \right>\right|. \ees

Dividing by $r$ and taking the average over $B_{r}(x)$ on both sides of (\ref{5}), we get 

\besn  \dashint_{B_{r}(x)} \frac{d(y, P_{x,r})}{r} \, d \mu(y)  &\leq & \, 2 \, \frac{1}{r} \,\, \dashint_{B_{r}(x)}  \left|  \left< y - \,\dashint_{B_{r}(x)} z \, d \mu(z)  , \nu_{x,2r} \right>\right|\, d \mu(y) \\
& \leq &  4 C_{P} \, \left(\, \dashint_{B_{2r}(x)} \left|\nu_{x,2r} - \nu(y)\right|^{2} d \mu\right)^{\frac{1}{2}},   \eesn
where the last inequality comes from (\ref{4}).\\

Thus, by the definition of $\alpha(x,2r)$ (see (\ref{a'})), we get (\ref{121}) and the proof is done.
\end{proof}

To start the proof of Theorem \ref{MTT'}, we want to use the construction of the bi-Lipschitz map given by David and Toro in their paper \cite{DT1}. For that, we need to introduce what we call a \textbf{coherent collection of balls and planes}. Here we follow the steps given by David and Toro (see \cite{DT1}, chapter 2). \\

First, set $r_{k} = 10^{-k-4}$ for $ k \in \mathbb{N}$, and let $\epsilon$ be a small number (will be chosen later) that depends only on $n$. Choose a collection $\{x_{jk} \}, \, \, j \in J_{k}$ of points in $\mathbb{R}^{n+1}$, so that

\be \label{59} |x_{jk} - x_{ik}| \geq r_{k} \,\,\,\,\, \txt{for}\,\,\, i,j \in J_{k}, \, i \neq j. \ee

Set $B_{jk} := B_{r_{k}}(x_{jk})$ and $V_{k}^{\lambda} := \displaystyle \bigcup_{j \in J_{k}} \lambda B_{jk} =  \displaystyle \bigcup_{j \in J_{k}}  B_{\lambda r_{k}}(x_{jk}),\,$ for $\lambda > 1$.\\

We also ask for our collection  $\{x_{jk} \}, \, \, j \in J_{k}$ and $k \geq 1$ to satisfy
\be \label{60} x_{jk} \in V_{k-1}^{2} \,\,\,\,\,\,\, \txt{for}\,\,\, k \geq 1 \,\,\, \txt{and} \,\,\, j \in J_{k}. \ee

Suppose that our initial net $\{x_{j0} \}$ is close to an $n$-dimensional plane $\Sigma_{0}$, that is
\be \label{101} d(x_{j0},\Sigma_{0}) \leq \epsilon \,\,\,\,\,\,\,\,\, \forall\, j \in J_{0}.\ee

For each $k \geq 0$ and $j \in J_{k}$, suppose you have an $n$-dimensional plane $P_{jk}$, passing through $x_{jk}$ such that the following compatibility conditions hold:\\

\be \label{102} d_{x_{i0},100r_{0}}(P_{i0}, \Sigma_{0}) \leq \epsilon \,\,\,\, \txt{for} \,\, i \in J_{0}, \ee

\be \label{74} d_{x_{ik},100r_{k}}(P_{ik},P_{jk}) \leq \epsilon \,\,\,\, \txt{for}\, k\geq 0 \,\,\,\txt{and} \,\,\, i,j \in J_{k} \,\,\, \txt{such that} \,\,\, |x_{ik} - x_{jk}| \leq 100r_{k},\ee

and

\be \label{75} d_{x_{ik},20r_{k}}(P_{ik},P_{j,k+1}) \leq \epsilon \,\,\, \txt{for}\, k\geq 0 \,\,\,\txt{and} \,\,\, i \in J_{k},\,j \in J_{k+1} \,\, \txt{such that} \,\,\, |x_{ik} - x_{j,k+1}| \leq 2r_{k}.\ee

We can now define a \textbf{coherent collection of balls and planes}:

\begin{definition}
A \textbf{coherent collection of balls and planes}, (in short a CCBP), is a triple $(\Sigma_{0}, \{B_{jk} \}, \{P_{jk}\})$ where the properties (\ref{59}) up to (\ref{75}) above are satisfied, with a prescribed $\epsilon$ that is small enough, and depends only on $n$.
\end{definition}

\begin{theorem} \label{t2}  (see Theorems 2.4 in \cite{DT1})
Let  $(\Sigma_{0}, \{B_{jk} \}, \{P_{jk}\})$ be a CCBP, and assume $\epsilon$ is small enough, depending on $n$. Then, there exists a bijection $g: \mathbb{R}^{n+1} \rightarrow \mathbb{R}^{n+1}$ with the following properties:
\be g(z)= z \quad \txt{when} \,\,\, d(z, \Sigma_{0}) \geq 2, \ee
and
\be |g(z)-z| \leq C_{0}^{'} \epsilon \quad \txt{for} \,\,\, z \in \mathbb{R}^{n+1}, \ee
where $C_{0}^{'}$ is a constant depending only on $n$. \\
Moreover, $g(\Sigma_{0})$ is a $C_{0}^{'} \epsilon$-Reifenberg flat set that contains the accumulation set 
\besn E_{\infty} = &\{x&  \in \mathbb{R}^{n+1}; \,\, x\,\, \txt{can be written as} \\ &x& = \lim_{m \to \infty} x_{j(m),k(m)}, \,\, \txt{with}\,\,k(m) \in \mathbb{N}, \\ &\txt{and}& \,\, j(m) \in J_{k_{m}} \,\, \txt{for}\,\, m \geq 0 \,\, \txt{and} \,\, \lim_{m \to \infty} k(m) = \infty\} .\eesn
\end{theorem}

In \cite{DT1}, David and Toro give a sufficient condition for $g$ to be bi-Lipschitz that we want to use in our proof. However, in order to state this condition, we need some technical details from the construction of the map $g$ from Theorem \ref{t2}. So, let us briefly discuss the construction here: David and Toro defined a mapping $f$ whose goal is to push a small neighbourhood of $\Sigma_{0}$ towards a final set, which they proved to be Reifenberg flat. They obtained $f$ as a limit of the composed functions $f_{k} = \sigma_{k-1} \circ \ldots \sigma_{0}$
where each $\sigma_{k}$ is a smooth function that moves points near the planes $P_{jk}$ at the scale $r_{k}$. More precisely, 
\be \label{sigmas} \sigma_{k} (y) = y + \sum_{j \in J_{k}} \theta_{jk}(y)[\pi_{jk}(y)-y], \ee
where $\{\theta_{jk}\}_{j \in J_{k}, k\geq 0}$ is a partition of unity with each $\theta_{jk}$ supported on $10B_{jk}$, and $\pi_{jk}$ denotes the orthogonal projection from $\mathbb{R}^{n}$ onto the plane $P_{jk}$.\\

\noindent Since $f$ in their construction was defined on $\Sigma_{0}$, $g$ was defined to be the extension of $f$ on the whole space.\\
 
\begin{corollary} \label{cr1} (see Proposition 11.2 in \cite{DT1}) Suppose we are in the setting of Theorem \ref{t2}. Define the quantity
\be  \label{nasty} \begin{split} \epsilon^{'}_{k}(y) &= \\ &sup\{d_{x_{im},100 r_{m}}(P_{jk},P_{im}); \,\,\, j \in J_{k}, \,\, i \in J_{m},\,\,\, m \in \{k,k-1\}, \,\,\txt{and}\,\, y \in 10B_{jk} \cap 11B_{im} \} \end{split} \ee
for $k \geq 1 \,\, \txt{and} \,\,y \in V_{k}^{10}$, and $\epsilon_{k}^{'}(y)=0 \,\, \txt{when}\,\, y \in \mathbb{R}^{n+1} \setminus V_{k}^{10}$ (when there are no pairs $(j,k)$ as above). 
If there exists $N > 0$ such that
 \be \label{89} \sum_{k=0}^{\infty} \epsilon^{'}_{k}(f_{k}(z)) ^{2} < N,  \ee
 then the map $g$ constructed in Theorem \ref{t2} is $K$-bi-Lipschitz, where the bi-Lipschitz constant $K$ depends only on $n$ and $N$.
\end{corollary}

We are finally ready to prove Theorem \ref{MTT'}.

\textbf{\underline{\textit{Proof of Theorem \ref{MTT'}:}}}

\begin{proof}
Let $\epsilon_{0} > 0$ (to be determined later), and suppose there is a choice of unit normal $\nu$ to $M$ such that (\ref{103}) holds. We would like to apply Theorem \ref{t2} and then Corollary \ref{cr1}. So our first goal is to construct a CCBP, and we do that in several steps:\\
 Let us start with a collection $\{\tilde{x}_{jk}\},\, j \in J_{k}$ of points in $M \cap B_{\frac{1}{10^{3}}}(0)$ that is maximal under the constraint \be \label{62}|\tilde{x}_{jk} - \tilde{x}_{ik}| \geq \displaystyle\frac {4r_{k}}{3} \,\,\,\,\txt{when}\,\, i,j \in J_{k}\,\,\, \txt{and}\,\,\,i \neq j.\ee
Of course, we can arrange matters so that the point $0$ belongs to our initial maximal set, at scale $r_{0}$. Thus, $0 = \tilde{x}_{i_{0},0} $ for some $i_{0} \in J_{0}$. Notice that for every $k \geq 0$, we have 
\be \label{63} M  \cap B_{\frac{1}{10^{3}}}(0)\subset \displaystyle \bigcup_{j \in J_{k}}\bar{B}_{\frac{4r_{k}}{3}}(\tilde{x}_{jk}).\ee
\\
Later, (see (\ref{68})), we choose \be \label{61} x_{jk} \in M \cap B_{ \frac{r_{k}}{6}}(\tilde{x}_{jk}), \,\,\,\,\,\, j \in J_{k} .\ee

By (\ref{63}) and (\ref{61}), we can see

\be \label{63'} M  \cap B_{\frac{1}{10^{3}}}(0)\subset \displaystyle \bigcup_{j \in J_{k}}\bar{B}_{\frac{4r_{k}}{3}}(\tilde{x}_{jk}) \subset \displaystyle \bigcup_{j \in J_{k}}B_{\frac{3r_{k}}{2}}(x_{jk}) .\ee

 Let us prove that such a collection $\{x_{jk} \},\,\,\, j \in J_{k}$ satisfies 
(\ref{59}) and (\ref{60}):\\

To see (\ref{59}), we proceed by contradiction. Suppose $|x_{jk} - x_{ik}| < r_{k}$ for some \,$i,j \in J_{k}, \, \, \txt{with} \,\,i\neq j$
Then, by (\ref{61}),
\ben | \tilde{x}_{jk} - \tilde{x}_{ik}| \leq |\tilde{x}_{jk} - x_{jk}| + |x_{jk} - x_{ik}| + |x_{ik} - \tilde{x}_{ik}| <  \frac{r_{k}}{6} + r_{k} +  \frac{r_{k}}{6} = \frac{4r_{k}}{3} \een
which contradicts (\ref{62}). This proves (\ref{59}). \\
\\

To see (\ref{60}), fix $x_{j,k+1} $ with $k\geq 0$ and $j \in J_{k+l}$. By construction and (\ref{63'}), we have
\be \label{61'}\tilde{x}_{j,k+1} \in M \cap B_{\frac{1}{10^{3}}}(0)\subset \displaystyle \bigcup_{i \in J_{k}}B_{\frac{3r_{k}}{2}}(x_{ik}) .\ee 

Using (\ref{61}) and (\ref{61'}), we get 

\ben x_{j,k+1} \in \displaystyle \bigcup_{i \in J_{k}}B_{2r_{k}}(x_{ik}) = V_{k}^{2}.\een
\\
Thus, (\ref{60}) is satisfied.\\

Next, we choose our planes $P_{jk}$ and our collection $\{ x_{jk} \}$, for $k \geq 0$ and $ j \in J_{k}$.
\\
Fix $k\geq 0$ and $j \in J_{k}$. Let $\epsilon_{1}$ be the constant from Lemma \ref{l4}. For
\be \label{c} \epsilon_{0} \leq \epsilon_{1}, \ee
we apply Theorem \ref{t1} to the point $\tilde{x}_{jk}$ (by construction  $\tilde{x}_{jk} \in M$) and radius $120r_{k}$ (notice that $120\,r_{k} \leq \frac{1}{20}$) to get an $n$-plane $P_{\tilde{x}_{jk},120r_{k}}$, denoted in this proof by $ P^{'}_{jk}$ for simplicity reasons, whose normal is $\nu_{\tilde{x}_{jk},240r_{k}}$ (recall from lemma \ref{l4} that $|\nu_{\tilde{x}_{jk},240r_{k}}| \geq \displaystyle \frac{1}{2}$) such that 

\be \label{66} \dashint_{B_{120r_{k}}(\tilde{x}_{jk})} \frac{d(y, P^{'}_{jk} )}{120r_{k}} \, d \mu  \leq  \, C(C_{P}) \, \alpha(\tilde{x}_{jk}, 240r_{k}). \ee

Thus, by (\ref{66}) and the fact that $\mu$ is Ahlfors regular, there exists $x_{jk} \in M \cap B_{ \frac{r_{k}}{6}}(\tilde{x}_{jk})$ such that 

\bes \label{68} d(x_{jk},P^{'}_{jk}) &\leq& \dashint_{B_{ \frac{r_{k}}{6}}(\tilde{x}_{jk})}d(y, P^{'}_{jk} ) \, d \mu \nonumber \\ &\leq&  C(n, C_{M}) \, \dashint_{B_{120r_{k}}(\tilde{x}_{jk})} d(y, P^{'}_{jk} ) \, d \mu \leq \, C(n, C_{M}, C_{P}) \, \alpha(\tilde{x}_{jk}, 240r_{k}) \,r_{k}. \ees

Let $P_{jk}$ be the plane parallel to $P^{'}_{jk}$ and passing through $x_{jk}$. Thus, $P_{jk}$ has normal line  $\nu_{\tilde{x}_{jk},240r_{k}}$ and passes through $x_{jk}$. From (\ref{68}) and the fact that the two planes are parallel, it is clear that 
\be \label{69} d_{\tilde{x}_{jk},240r_{k}}(P_{jk},P^{'}_{jk}) \leq C(n, C_{M}, C_{P})\,  \alpha(\tilde{x}_{jk}, 240r_{k}). \ee
\\
Moreover, for every $y \in B_{120r_{k}}(\tilde{x}_{jk})$, we have by the triangle inequality and (\ref{69})
\bes \label{70} d(y,P_{jk}) &\leq& d(y,P^{'}_{jk}) + c \, d_{\tilde{x}_{jk},240r_{k}}(P_{jk},P^{'}_{jk})\, r_{k} \nonumber \\
&\leq& d(y,P^{'}_{jk}) + C(n, C_{M}, C_{P})\,  \alpha(\tilde{x}_{jk}, 240r_{k}) \, r_{k}.\ees

Dividing both sides of (\ref{70}) by $120r_k$ and taking the average over $B_{120r_{k}}(\tilde{x}_{jk})$, we get 
\be \label{72} \dashint_{ B_{120r_{k}}(\tilde{x}_{jk})} \frac{d(y,P_{jk})}{120r_{k}} \, d \mu \leq  \dashint_{ B_{120r_{k}}(\tilde{x}_{jk})} \frac{d(y,P^{'}_{jk})}{120r_{k}} \, d \mu+  C(n, C_{M}, C_{P})\,  \alpha(\tilde{x}_{jk}, 240r_{k}), \ee
which by (\ref{66}) becomes 
\be \label{73} \dashint_{ B_{120r_{k}}(\tilde{x}_{jk})} \frac{d(y,P_{jk})}{120r_{k}} \, d \mu \leq  C(n, C_{M}, C_{P})\,  \alpha(\tilde{x}_{jk}, 240r_{k}). \ee

To summarize what we did so far, we have chosen $n$-dimensional planes $P_{jk}$ for $k\geq 0$ and $j \in J_{k}$ where each $P_{jk}$ has normal line is $\nu_{\tilde{x}_{jk},240r_{k}}$, passes through $x_{jk}$, and satisfies (\ref{73}).\\

We proceed by proving (\ref{102}), (\ref{74}), and (\ref{75}), starting with (\ref{74}) and (\ref{75}) .\\

We prove (\ref{74}) and (\ref{75}) simultaneously here. So, let us fix $k \geq 0$ and $j \in J_{k}$; let $m \in \{k, k-1 \}$ and $i \in J_m$ such that 
\be \label{76} |x_{jk} - x_{im}| \leq 100r_{m}.\ee

We want to show that $P_{jk}$ and $P_{im}$ are close together. To do that, we construct $n$ linearly independent vectors that ``effectively'' span $P_{jk}$, that is, these vectors span $P_{jk}$, and they are far away from each other (in a uniform quantitative manner).  Then, we show that $P_{im}$ is close to each of these vectors. This idea is very similar to the ``effectively'' spanning idea found in \cite{NV1} (see p. 26-28).\\
Let us start by proving the existence of such vectors in the following claim. Here is where we use lemma \ref{l1}.\\
\\\
 \textbf{\textit{Claim:}} 
Denote by  $\pi_{jk}$ is the orthogonal projection of $\mathbb{R}^{n+1}$ on the plane $P_{jk}$. Let $r = c_{0} \, r_{k}$, where $c_{0} \leq \displaystyle \frac{1}{2}$ is the constant from Lemma \ref{l1} depending only on $n$ and $C_{M}$. Then, there exists a sequence of $n+1$ balls $\{B_{r}(y_{l})\}_{l=0}^{n}$,  such that 
\begin{enumerate}
\item $\forall \, l \in \{ 0, \ldots n\}$, we have $y_{l} \in M$ and $B_{r}(y_{l}) \subset B_{2r_{k}}(\tilde{x}_{jk}).$
\item  $q_{1} - q_{0} \notin B_{5r}(0)$, and $\forall \,  l \in \{ 2, \ldots n\}$, we have $q_{l} - q_{0} \notin N_{5r}\big(span \{q_{1} - q_{0}, \ldots, q_{l-1} - q_{0}  \}\big),$
\end{enumerate}
where $q_{l} = \pi_{jk}(p(y_{l}))$ and $p(y_{l}) = \dashint_{B_{r}(y_{l})}z \, d\mu(z)$ is the centre of mass of $\mu$ in the ball $B_{r}(y_{l})$.\\

We prove this claim by induction:\\

For $l=0$, take $y_{0} = \tilde{x}_{jk}$ (recall that both $k$ and $j$ are fixed here). In this case, item 1 is trivial, and item 2 is not applicable. Thus, we have our points $y_{0}, \, p(y_{0}), \,$ and $q_{0}$.\\
\\
Now, let $r = c_{0} \, r_{k}$ as in Lemma \ref{l1}, where we have applied the lemma on $x_{0} =\tilde{x}_{jk}$ and $r_{0} = r_{k}$. Recall that the constant $c_{0}$ we get from Lemma \ref{l1} is as desired (that is $c_{0} \leq \displaystyle \frac{1}{2}$ depending only on $n$ and $C_{M}$.)
For $i =1$, we apply Lemma \ref{l1} for $V = \{\tilde{x}_{jk}\}$, to get a point $y_{1} \in M \cap B_{r_{k}}(\tilde{x}_{jk})$ such that $y_{1} \notin B_{11\,r}(\tilde{x}_{jk})$ and $B_{r}(y_{1}) \subset B_{2r_{k}}(\tilde{x}_{jk})$. So item 1 is satisfied, and now we have our points $p(y_{1})$ and $q_{1}$.\\ 
\\
For item 2, we need to prove that 
\be \label{38} |q_{1} - q_{0}| \geq 5r.\ee
In fact, we have for $l \in \{0,1\}$, by the definition of $p(y_{l})$, Jensen's inequality applied on the convex function $\phi(.) = d(.,P_{jk})$, the fact that $\mu$ is Ahlfors regular, $B_{r}(y_{l}) \subset B_{2r_{k}}(\tilde{x}_{jk})$, $r = c_{0}\, r_{k}$, and (\ref{73}), that 

\bes \label{33} d\big(p(y_{l}),P_{jk}\big) &=& d\bigg( \, \dashint_{B_{r}(y_{l})}z \, d\mu(z), P_{jk} \bigg) \nonumber \\
&\leq& \dashint_{B_{r}(y_{l})} d(z, P_{jk} ) \, d \mu(z) \nonumber \\ 
&\leq& C(n, C_{M}) \, \dashint_{B_{120\,r_{k}}(\tilde{x}_{jk})} d(z, P_{jk} ) \, d \mu(z) 
\leq  C(n, C_{M}, C_{P}) \, \alpha(\tilde{x}_{jk},240 \,r_{k}) \, r_{k}.\ees

Also, by the definition of the center of mass, we know that 
\be \label{34} |y_{l} - p(y_{l})| \leq r \,\,\,\,\,\,\,\,\,\, l \in \{0,1\}. \ee

Thus, by the triangle inequality, (\ref{34}), and (\ref{33}), we get for $\,l \in \{0,1\}$
\bes \label{35''} |y_{l}-q_{l}| &\leq& |y_{l}- p(y_{l})| + |p(y_{l})-q_{l}| \nonumber\\
&=& |y_{l}- p(y_{l})| + d\big(p(y_{l}),P_{jk}\big) \nonumber\\
&\leq& r + C(n, C_{M}, C_{P})\, \alpha(\tilde{x}_{jk},240 \,r_{k}) r_{k}.\ees

Notice now, that by (\ref{103}), (\ref{a''}) in Lemma \ref{l5}, the fact that $\tilde{x}_{jk} \in M \cap B_{\frac{1}{10^{3}}}(0)$ and $240r_{k} \displaystyle \leq \frac{1}{10}$, we have 
\be \label{132}  \alpha(\tilde{x}_{jk},240r_{k}) \leq C(n, C_{M}) \, \epsilon_{0}. \ee

Plugging (\ref{132}) in (\ref{35''}), and using the fact that $r = c_{0}\, r_{k}$,  we get for $\,l \in \{0,1\}$
\be \label{35'} |y_{l}-q_{l}|
\leq r + C(n, C_{M}, C_{P})\,\epsilon_{0} \,r_{k} 
=  r + C(n, C_{M}, C_{P})\, \epsilon_{0} \, r.\ee

Let us denote by $C_{1}$ the constant $C(n, C_{M}, C_{P})$ from the last step of (\ref{35'}). Then, rewriting (\ref{35'}), we get
$ |y_{l}-q_{l}| \leq  r + C_{1}\, \epsilon_{0} \, r $. For $\epsilon_{0}$ such that $ C_{1} \,\epsilon_{0} < 1$, we get

\be \label{37} |y_{l}-q_{l}| \leq 2r \,\,\,\,\,\,\,\,\,\,\,\, l \in \{0,1\}. \ee

We are now ready to prove (\ref{38}):\\

Let us proceed by contradiction. Suppose that $ |q_{1} - q_{0}| < 5r$, then by (\ref{37}), we get
\besn |y_{1} - y_{0}| &\leq& |y_{1} - q_{1}| +  |q_{1} - q_{0}| + |y_{0} - q_{0}| \\
&\leq& 2r + 5r + 2r = 9r.\eesn
But $y_{1} \notin B_{11r}(\tilde{x}_{jk}) = B_{11r}(y_{0})$ by construction. Thus, we get a contradiction, and (\ref{38}) is proved.\\
\\
For our induction step, assume the statement is true for $l-1$, and let's prove it for $l$. Consider the $(l-1)$-dimensional affine subspace 
\ben V^{l-1} = span \{ q_{1} - q_{0}, \ldots q_{l-1} - q_{0}\} + q_{0}. \een

Notice that our last induction process is when we have $n$ points and want to construct the $(n+1)^{st}$ point. Thus, $l-1 \leq n-1$, and we can apply Lemma \ref{l1}, on the subspace $V^{l-1}$, to get a point $y_{l} \in M \cap B_{r_{k}}(\tilde{x}_{jk})$ such that $y_{l} \notin N_{11\,r}(V^{l-1})$. 
So, we have that
\be \label{39} y_{l} - q_{0} \notin   N_{11r}\big( span \{ q_{1} - q_{0}, \ldots q_{l-1} - q_{0}\}\big).\ee
Item 1 is clearly true. To prove item 2, we show that
\be \label{40} q_{l} - q_{0} \notin  N_{5r}\big( span \{ q_{1} - q_{0}, \ldots q_{l-1} - q_{0}\}\big). \ee

In fact, by the exact same calculations as above \big(see (\ref{33}), (\ref{34}), and (\ref{37})\big), we see that

\be \label{41} d\big(p(y_{l}),P_{jk}\big) \leq  C(n, C_{M}, C_{P}) \, \alpha(\tilde{x}_{jk},240\,r_{k})\,r_{k},\ee

\be \label{42} |y_{l}-p(y_{l})| \leq r,  \ee
and

\be \label{43} |y_{l}-q_{l}| \leq 2r .\ee

Let us now prove (\ref{40}) by contradiction:\\
Suppose that $q_{l} - q_{0} \in  N_{5r}\big( span \{ q_{1} - q_{0}, \ldots q_{l-1} - q_{0}\}\big)$, then, using (\ref{43}), we get

\ben \begin{split}
d\big(y_{l} - q_{0}, span \{ q_{1} - q_{0}, \ldots &, q_{l-1} - q_{0}\}\big ) \\ &\leq d(y_{l} - q_{0}, q_{l}-q_{0}) + d\big(q_{l} - q_{0}, span \{ q_{1} - q_{0}, \ldots q_{l-1} - q_{0}\}\big )\\
&= |y_{l}-q_{l}|+ d\big(q_{l} - q_{0}, span \{ q_{1} - q_{0}, \ldots q_{l-1} - q_{0}\}\big )\\
&\leq 2r+5r = 7r < 11r .\end{split}
\een

which is a contradiction by (\ref{39}). Thus, induction process is complete, and so is the proof of the claim $\blacksquare$\\

From the construction in the claim above, notice that 
\be \label{78} P_{jk} - q_{0} = span  \{q_{1} - q_{0} , \ldots, q_{n} - q_{0}  \} .\ee

Also, by (\ref{33}) and (\ref{41}), we have  $\forall \, l \in \{ 0, \ldots n\}$
\be \label{77} d\big(p(y_{l}),P_{jk}\big) \leq  C(n, C_{M}, C_{P}) \, \alpha(\tilde{x}_{jk},240r_{k})\,r_{k},\ee

and by (\ref{37}), and recalling that $y_{0} = \tilde{x}_{jk}$ we have
\be \label{86} |y_{0}-q_{0}| = |\tilde{x}_{jk}- q_{0}| \leq 2r .\ee
\\
Let us remember that our goal is to prove that $P_{jk}$ and $P_{im}$ are close to each other. In the claim, we constructed an ``effective'' spanning set for $P_{jk}$, $ \{q_{1} - q_{0}, \ldots, q_{n} - q_{0}  \}  $. Now, we can get a nice upper bound on the distance from each $q_{l}$ to $P_{im}$, for $\, l \in \{ 0, \ldots n\}$.\\

In fact, by the definition of the center of mass, Jensen's formula, the fact that $\mu$ is Ahlfors regular, $B_{r}(y_{l}) \subset B_{120r_{m}}(\tilde{x}_{im})$ (see item 1, (\ref{76}), (\ref{61}), and recall that $r < r_{k} \leq r_m$), $r = c_{0} \, r_{k}$ and $r_{m} \in \{ r_{k}, 10r_{k}\}$, and (\ref{73}) for $P_{im}$, to get that for every $ \, l \in \{ 0, \ldots n\}$

\bes \label{79} d\big(p(y_{l}), P_{im}\big) &=& d\bigg( \, \dashint_{B_{r}(y_{l})}z \, d\mu(z),P_{im}\bigg) \nonumber \\
&\leq& \dashint_{B_{r}(y_{l})} d(z, P_{im} ) \, d \mu(z) \nonumber \\ 
&\leq& C(n, C_{M})\dashint_{B_{120r_{m}}(\tilde{x}_{im})} d(z, P_{im} ) \, d \mu(z)
\leq  C(n, C_{M}, C_{P}) \, \alpha(\tilde{x}_{im},240r_{m}) \, r_{m}.\ees

Combining (\ref{77}) and (\ref{79}), we get by the triangle inequality that for every $ \, l \in \{ 0, \ldots n\}$

\bes \label{80} d\big(q_{l},P_{im}\big) &\leq& |q_{l} - p(y_{l})| + d\big(p(y_{l}),P_{im}\big) \nonumber \\
&=&  d\big(p(y_{l}), P_{jk}\big) + d\big(p(y_{l}),P_{im}\big) \nonumber \\
&\leq& C(n, C_{M}, C_{P}) \, \big( \alpha(\tilde{x}_{jk},240r_{k})\,r_{k} + \alpha(\tilde{x}_{im},240r_{m}) \, r_{m}\big).
\ees

We are finally ready to compute the distance between $P_{jk}$ and $P_{im}$. Let $y \in P_{jk} \cap B_{\rho}(x_{im})$ where $\rho \in \{20r_m, 100r_{m}\}$.
By (\ref{78}), $y$ can be written uniquely as $ y -q_{0} = \displaystyle \sum_{l=1}^{n} \beta_{l}(q_{l} - q_{0})$, that is
\be \label{82} y  = q_{0} + \sum_{l=1}^{n} \beta_{l}(q_{l} - q_{0}).\ee 

We want to apply Lemma \ref{l3}, for $u_{l} = q_{l} - q_{0}$, $R = r$, and $v = y - q_{0}$. In fact, (\ref{44}) and (\ref{45}) are satisfied directly from item 2 for $k_{0} = 5$. To see (\ref{44'}), note that by (\ref{43}), (\ref{37}), the fact that $y_{0}$ and $y_{l} \in B_{2r_{k}}(\tilde{x}_{jk})$ from item 1, for $ l \in \{ 1, \ldots , n \}$, and $r = c_{0} \, r_{k}$, we have
\be \label{mm''} |q_{l} - q_{0}| \leq  |q_{l} - y_{l}| + |y_{l} - y_{0}| + |y_{0} - q_{0}| \leq 4r + 2r_{k} \leq C_{2}\, r ,\ee
where $C_{2}$ is a (fixed) constant depending only on $n$ and $C_{M}$. \\

For $K_{0} = C_{2}$ where $K_{0}$ the constant in the statement of Lemma \ref{l3}, we get by Lemma \ref{l3} that
\be \label{84} |\beta_{l}| \leq K_{1} \, \frac{1}{r}\,|y-q_{0}| \,\,\,\,\, \forall \, l \in \{ 1, \ldots n\}.\ee

However, by (\ref{76}), (\ref{61}), (\ref{86}), and remembering that $r< r_{k} \leq r_{m}$, we have

\bes \label{85} |y- q_{0}| &\leq& |y- x_{im}| + |x_{im} - x_{jk}| + |x_{jk} - \tilde{x}_{jk}| + |\tilde{x}_{jk} - q_{0}| \nonumber\\
&\leq& \rho + 100 r_{m} +  \frac{r_{k}}{6} + 2r \leq 203 r_m .
 \ees

But $r = c_{0}\, r_{k}$, and $r_{m} = \{r_{k}, 10r_{k}\}$, and thus, combining (\ref{84}) and (\ref{85}), we get
\be \label{87}  |\beta_{l}| \leq C(n, C_{M}). \ee

So, using (\ref{82}), (\ref{87}), and (\ref{80}), we get
\bes  d\big(y, P_{im}\big) &\leq& (1 +  \sum_{l=1}^{n} |\beta_{l}|) \,d\big(q_{0}, P_{im}\big) + \sum_{l=1}^{n} |\beta_{l}| \, d\big(q_{l}, P_{im}\big) \nonumber\\ 
&\leq& C(n, C_{M})\, \bigg( d\big(q_{0}, P_{im}\big) + \sum_{l=1}^{n}   d\big(q_{l}, P_{im}\big) \bigg) \nonumber\\ 
&\leq& C(n, C_{M}, C_{P}) \, \bigg( \alpha(\tilde{x}_{jk},240r_{k})\,r_{k} + \alpha(\tilde{x}_{im},240r_{m}) \, r_{m}\bigg)
\ees

Thus, 
\be \label{88}d_{x_{im}, \rho} (P_{jk}, P_{im}) \leq  C(n, C_{M}, C_{P})\, \bigg( \alpha(\tilde{x}_{jk},240r_{k}) + \alpha(\tilde{x}_{im},240r_{m})\bigg)\,\,\,\,\, \rho \in\{20r_m, 100r_{m}\}. \ee
And so, our planes $P_{jk}$ and $P_{im}$ are close. In fact, by (\ref{132}), we know that  
\be \label{133} \alpha(\tilde{x}_{jk},240r_{k}) \leq C(n, C_{M})\, \epsilon_{0}. \ee
Similarly, we have 
\be \label{134} \alpha(\tilde{x}_{im},240r_{m}) \leq C(n, C_{M}) \, \epsilon_{0}. \ee
Plugging (\ref{133}) and (\ref{134}) in (\ref{88}), we get 
\be \label{135} d_{x_{im}, \rho} (P_{jk}, P_{im}) \leq  C(n, C_{M}, C_{P}) \epsilon_{0}, \quad \rho \in\{20r_m, 100r_{m}\}. \ee

So, we have shown that there exists two constants $C_{3}$ and $C_{4}$, each depending on $n$, $C_{M}$ and $C_{P}$, such that
\be \label{74'} d_{x_{ik},100r_{k}}(P_{ik},P_{jk}) \leq C_{3} \,\epsilon_{0} \,\,\,\, \txt{for}\, k\geq 0 \,\,\,\txt{and} \,\,\, i,j \in J_{k} \,\,\, \txt{such that} \,\,\, |x_{ik} - x_{jk}| \leq 100r_{k},\ee

and

\be \label{75'} d_{x_{ik},20r_{k}}(P_{ik},P_{j,k+1}) \leq C_{4} \,\epsilon_{0} \,\,\, \txt{for}\, k\geq 0 \,\,\,\txt{and} \,\,\, i \in J_{k},\,j \in J_{k+1} \,\, \txt{such that} \,\,\, |x_{ik} - x_{j,k+1}| \leq 2r_{k}.\ee

For 
\be \label{c'''} C_{3} \,\epsilon_{0} \leq \epsilon \quad \txt{and} \quad C_{4} \,\epsilon_{0} \leq \epsilon, \ee
we get (\ref{74}) and (\ref{75}). \\
\\
We now prove (\ref{102}). Recall that $0 = \tilde{x}_{i_{0},0}$ for some $i_{0} \in J_{0}$. Choose $\Sigma_{0}$ to be the plane $P_{i_{0},0}$ described above (that is $P_{i_{0},0}$ has normal $\nu_{\tilde{x}_{i_{0},0},240r_{0}}$ and passes through $x_{i_{0},0}$, where $r_{0} = 10^{-4}$). Then, what we need to prove is 

\be \label{102'} d_{x_{j0},100r_{0}}(P_{j0}, P_{i_{0},0}) \leq \epsilon \,\,\,\, \txt{for} \,\, j \in J_{0}. \ee

Fix $j \in J_{0}$, and take the corresponding $x_{j0}$. Since by construction $|\tilde{x}_{j0}| < \displaystyle \frac{1}{10^{3}}$ and since (\ref{61}) says that $ |x_{j_{0},0} -  \tilde{x}_{j_{0},0}| \leq \displaystyle \frac{r_{0}}{6}$, then, we have
\be \label{1''}|x_{j0}| \leq \frac{r_{0}}{6} + \frac{1}{10^{3}},\,\,\,\,\,\, j \in J_{0}.\ee
Moreover, by (\ref{61}) and the fact that $0 = \tilde{x}_{i_{0},0}$ , we have
\be \label{2''} |x_{i_{0},0} -  \tilde{x}_{i_{0},0}| = |x_{i_{0},0}| \leq \frac{r_{0}}{6}.\ee
Combining (\ref{1''}) and (\ref{2''}), and using the fact that $r_{0} = 10^{-4}$ we get
\be |x_{j0} - x_{i_{0},0} | \leq \frac{r_{0}}{6} +  \frac{1}{10^{3}} + \frac{r_{0}}{6} \leq \frac{r_{0}}{6} +  10 r_{0} + \frac{r_{0}}{6} \leq 100r_{0}. \ee
Thus, by (\ref{74}) for $x_{ik} = x_{j0}$, $P_{ik} = P_{j0}$, and $P_{jk} = P_{i_{0},0}$, we get exactly (\ref{102'}), hence finishing the proof for (\ref{102}).\\
\\
It remains to prove (\ref{101}), that is 
\be \label{101'} d(x_{j0}, P_{i_{0},0}) \leq \epsilon,\quad \txt{for}\,\, j \in J_{0}. \ee

By Markov's inequality, we know that
\ben \label{3''} \mu \left ( x \in B_{120r_{0}}(\tilde{x}_{i_{0},0}); d(x,P_{i_{0}0})\geq \alpha^{\frac{1}{2}} (\tilde{x}_{i_{0},0}, 240r_{0}) \right) \leq \frac{1}{ \alpha^{\frac{1}{2}}(\tilde{x}_{i_{0},0}, 240r_{0})} \int_{B_{120r_{0}}(\tilde{x}_{i_{0},0})} d(y,P_{i_{0}0}) \een

But since $\tilde{x}_{i_{0},0} = 0$, and by using (\ref{73}) with the fact that $\mu$ is Ahlfors regular, and (\ref{103}) with (\ref{a''}) from Lemma \ref{l5} and the fact that $240r_{0} \leq \frac{1}{10}$, we get 

\bes \label{4''}   \mu \left(x \in B_{120r_{0}}(0) ; \,\,\, d(x,P_{i_{0}0}) \geq \alpha^{\frac{1}{2}}(0,240r_{0})\right) &\leq& \frac{1}{ \alpha^{\frac{1}{2}}(0,240r_{0})} \int_{B_{120r_{0}}(0)} d(y,P_{i_{0}0})\, d \mu \nonumber \\
&=& \frac{\mu(B_{120r_{0}}(0))}{ \alpha^{\frac{1}{2}}\left(0,240r_{0}\right)} \,\, \dashint_{B_{120r_{0}}(0)} d(y,P_{i_{0}0}) \nonumber \\
&\leq& C(n, C_{M}, C_{P}) \,  \alpha^{\frac{1}{2}}\left(0,240r_{0}\right) \nonumber \\ &\leq& C(n, C_{M}, C_{P})\, \epsilon_{0} ^{\frac{1}{2}}.
\ees

Now, take a point $z \in M \cap B_{120r_{0}}(0)$. We consider two cases:\\
Either 
\be \label{5'''} d(z ,P_{i_{0}0}) \leq \alpha^{\frac{1}{2}}\left(0, 240r_{0}\right) \ee
or
\be \label{5'} d(z ,P_{i_{0}0}) > \alpha^{\frac{1}{2}}\left(0, 240r_{0}\right). \ee

In the first case, combining (\ref{5'''}) with (\ref{103}) and (\ref{a''}), we get 

\be \label{6} d(z ,P_{i_{0}0}) \leq C(n, C_{M}) \,\epsilon_{0}^{\frac{1}{2}} .\ee

In case of (\ref{5'}), let $\rho$ be the biggest radius such that \ben B_{\rho}(z) \subset  \left\{ x \in B_{120r_{0}}(0) ; \,\,\, d(x,P_{i_{0}0}) > \alpha^{\frac{1}{2}}\left(0, 240r_{0}\right)\right\}. \een

Now, since $z \in M$ and $\mu$ is Ahlfors regular, we get using (\ref{4''}) that

\be \label{8''} C_{M} \, \rho^{n} \leq \mu(B_{\rho}(z)) \leq C(n, C_{M}, C_{P}) \,\epsilon_{0}^{\frac{1}{2}}. \ee

Thus, relabelling, (\ref{8''}) becomes

\be \label{9''} \rho \leq C(n, C_{M}, C_{P}) \, \epsilon_{0}^{\frac{1}{2n}}.\ee

On the other hand, since $\rho$ is the biggest radius such that $B_{\rho}(z) \subset \\ \left\{ x \in B_{120r_{0}}(0) ; d(x,P_{i_{0}0}) > \alpha^{\frac{1}{2}}\left(0, 240r_{0}\right)\right\}$, then there exists $x_{0} \in \partial B_{\rho}(z)$ such that 
\be \label{5''} d(x_{0} ,P_{i_{0}0}) \leq \alpha^{\frac{1}{2}}\left(0, 240r_{0}\right). \ee

Thus, by (\ref{5''}), (\ref{9''}) and (\ref{103}) together with (\ref{a''}), we get

\bes \label{7''} d(z ,P_{i_{0}0}) &\leq& |z-x_{0}| +  d(x_{0} ,P_{i_{0}0}) \nonumber \\
&=& \rho +  d(x_{0} ,P_{i_{0}0}) \leq  C(n, C_{M}, C_{P}) \, \epsilon_{0}^{\frac{1}{2n}}+ \alpha^{\frac{1}{2}}\left(0, 240r_{0}\right) 
\leq  C(n, C_{M}, C_{P})  \epsilon_{0}^{\frac{1}{2n}}. \ees

Combining (\ref{6}) and (\ref{7''}), we get that 
\be  \label{10''''}  d(z ,P_{i_{0}0}) \leq C_{5} \,  \epsilon_{0}^{\frac{1}{2n}} \quad \txt{for}\,\, z \in M \cap B_{120r_{0}}(0), \ee
where $C_{5}$ is a (fixed) constant depending only on $n$, $C_M$, and $C_{P}$.

We are now ready to prove (\ref{101'}). Fix $j \in J_{0}$, and take the corresponding $x_{j0}$. Since by construction $|\tilde{x}_{j0}| < \displaystyle \frac{1}{10^{3}}$ and since (\ref{61}) says that $ |x_{j_{0},0} -  \tilde{x}_{j_{0},0}| \leq \displaystyle \frac{r_{0}}{6}$, then, remembering that $r_{0} = 10^{-4}$, we have
\ben |x_{j0}| \leq \frac{r_{0}}{6} + \frac{1}{10^{3}} \leq 11 r_{0},\quad j \in J_{0}. \een
Thus,
\be \label{1'} x_{j0} \in M \cap B_{11r_{0}}(0) \subset M \cap B_{120r_{0}}(0).  \ee

For $z = x_{j0}$ in (\ref{10''''}), and for $ C_{5} \,  \epsilon_{0}^{\frac{1}{2n}} \leq \epsilon,$
we get

\be \label{11'}  d(x_{j0} ,P_{i_{0}0}) \leq \epsilon \,\,\,\,\,\,\,\,j \in J_{0}, \ee
which is exactly (\ref{101'}).\\

Fix $\epsilon_{0}$ such that (\ref{c}), (\ref{c'''}), the line before (\ref{37}), and the line before (\ref{11'}) are all satisfied. Then, we finally have our CCBP. Now, by the proof of Theorem \ref{t2} (see paragraph above (\ref{sigmas})) we get the smooth maps $\sigma_{k} \,\, \txt{and} \,\, f_{k} = \sigma_{k-1} \circ \ldots \sigma_{0} \,\, \txt{for} \,\, k \geq 0$, and then the map $f = \displaystyle \lim_{k \to \infty} f_{k}$ defined on $\Sigma_{0}$, and finally the map $g$ that we want.
 
Moreover, by Theorem \ref{t2}, we know that $g: \mathbb{R}^{n+1} \rightarrow \mathbb{R}^{n+1}$ is a bijection with the following properties:
\be \label{aa1} g(z)= z \,\,\, \,\,\, \txt{when} \,\,\, d(z, \Sigma_{0}) \geq 2, \ee
\be \label{bb1} |g(z)-z| \leq C_{0}^{'} \epsilon \,\,\,\,\,\,\, \txt{for} \,\,\, z \in \mathbb{R}^{n+1}, \ee
and \be \label{cc1} g(\Sigma_{0}) \,\, \txt{is a } \,\,\, C_{0}^{'} \epsilon \txt{-Reifenberg flat set}. \ee 

Notice that by the choice of $\epsilon_{0}$, we can write 
 $\epsilon_{0} = c_{6}\, \epsilon$, 
where $c_{6}$ is a constant depending only on $n$, $C_{M}$, and $C_{P}$. Hence, from (\ref{aa1}), (\ref{bb1}), (\ref{cc1}), we directly get (\ref{aa}), (\ref{bb}), and (\ref{cc}).
\\

We next show that 
\be \label{12''} M \cap B_{\frac{1}{10^{3}}}(0) \subset  g(\Sigma_{0}). \ee Fix $x \in  M \cap B_{\frac{1}{10^{3}}}(0) $. Then, by (\ref{63'}), we see that for all $k \geq 0$, there exists a point $x_{jk}$ such that $|x - x_{jk}| \leq \displaystyle \frac{3r_{k}}{2}$, and hence $x \in E_{\infty} \subset g(\Sigma_{0})$ ($E_{\infty}$ is the set defined in Theorem \ref{t2}). Since $x$ was an arbitrary point in $M \cap B_{\frac{1}{10^{3}}}(0)$, (\ref{12''}) is proved.\\

We still need to show that $g$ is bi-Lipschitz. By Corollary \ref{cr1}, it suffices to show (\ref{89}).In order to do that, we need the following inequality from \cite{DT1} (see inequality (6.8) page 27 in \cite{DT1} \footnote{Inequality (6.8) in \cite{DT1} has a $C\epsilon$ in front of $r_{k}$; however, $\epsilon$ was later chosen so that $C \epsilon \leq 1$ which gives us our inequality (\ref{in}) above.}).\\

\be \label{in} |f(z) - f_{k}(z)| \leq r_{k} \quad \txt{for}\,\, k \geq 0\,\, \txt{and} \,\, z \in \Sigma_{0}.\ee

Let $z \in \Sigma_{0}$, and choose $\bar{z} \in M \cap B_{\frac{1}{10^{3}}}(0)$ such that 
\be \label{90} |\bar{z} - f(z)| \leq 2 \, d(f(z),M \cap B_{\frac{1}{10^{3}}}(0)). \ee
Fix $k \geq 0$, and consider the index $m  \in \{k,k-1\}$ and the indices $j \in J_{k}$ and $i \in J_{m}$ such that $f_{k}(z) \in 10B_{jk} \cap 11B_{im}$. We show that
\be \label{91}  d_{x_{im},100 r_{m}}(P_{jk},P_{im}) \leq C(n, C_{M}, C_{P}) \, \alpha(\bar{z},r_{k-4}) \quad \txt{for} \,\, k \geq 1. \ee

Notice that by (\ref{90}) and (\ref{in}), and since $\tilde{x}_{jk} \in M \cap B_{\frac{1}{10^{3}}}(0)$, $|\tilde{x}_{jk} - x_{jk}| \leq \displaystyle  \frac{r_{k}}{6}$, and $f_{k}(z) \in 10B_{jk}$, we have 
\bes \label{94} |\bar{z} - f_{k}(z)| &\leq& |\bar{z} - f(z)| + |f(z) - f_{k}(z)| \nonumber \\
&\leq& 2 \,d(f(z),M \cap B_{\frac{1}{10^{3}}}(0)) + |f(z) - f_{k}(z)| \nonumber \\
&\leq& 2 \, d(f_{k}(z),M \cap B_{\frac{1}{10^{3}}}(0)) + 3|f(z) - f_{k}(z)| \nonumber \\
&\leq& 2 \, |f_{k}(z)-\tilde{x}_{jk}| + 3r_{k} \nonumber \\
&\leq& 2 \, |f_{k}(z)-x_{jk}| + |\tilde{x}_{jk} - x_{jk}|+ | + 3r_{k} \nonumber \\
&\leq&20r_{k} +  \frac{r_{k}}{6} + 3r_{k} \leq 24r_{k}.
\ees

Thus, \be \label{93} B_{240r_{k}}(\tilde{x}_{jk}) \subset B_{r_{k-4}}(\bar{z}). \ee
In fact, for $a \in B_{240r_{k}}(\tilde{x}_{jk}) $, we have by (\ref{61}), the fact that $f_{k}(z) \in 10B_{jk}$, and (\ref{94}), that 
\besn |a - \bar{z}| &\leq& |a - \tilde{x}_{jk}| + |\tilde{x}_{jk} - x_{jk}| + |x_{jk} - f_{k}(z)| + |f_{k}(z) - \bar{z}| \\
&\leq& 240r_{k} +  \frac{r_{k}}{6} + 10r_{k} + 24r_{k} \leq r_{k-4}. \eesn

Similarly, we can show that 
 \be \label{95} B_{240r_{m}}(\tilde{x}_{im}) \subset B_{r_{k-4}}(\bar{z}). \ee

Thus, by (\ref{93}) and (\ref{95}), we have
 \be \label{96} B_{240r_{m}}(\tilde{x}_{im}) \cup  B_{240r_{k}}(\tilde{x}_{jk}) \subset B_{r_{k-4}}(\bar{z}). \ee

But, using (\ref{10}) for $ a = \nu_{\bar{z},r_{k-4}}$, (\ref{96}), and the fact that $\mu$ is Ahlfors regular
\besn \alpha^{2}(\tilde{x}_{jk}, 240r_{k}) &=& \dashint_{B_{240r_{k}}(\tilde{x}_{jk})} |\nu(y) - \nu_{\tilde{x}_{jk},240r_{k}}|^{2}\, d \mu\\ 
&\leq& \dashint_{B_{240r_{k}}(\tilde{x}_{jk})}  |\nu(y) - \nu_{\bar{z},r_{k-4}}|^{2}\, d \mu \\
&\leq& C(n, C_{M})\, \dashint_{B_{r_{k-4}}(\bar{z})}  |\nu(y) - \nu_{\bar{z},r_{k-4}}|^{2}\, d \mu 
= C(n, C_{M})\, \alpha^{2}(\bar{z},r_{k-4}),
\eesn 
and thus,
\be \label{97}  \alpha(\tilde{x}_{jk}, 240r_{k})  \leq C(n, C_{M})\,  \alpha(\bar{z},r_{k-4}). \ee

Similarly, we can show that 
\be \label{98}  \alpha(\tilde{x}_{im}, 240r_{m})  \leq C(n, C_{M})\,  \alpha(\bar{z},r_{k-4}). \ee

Plugging (\ref{97}) and (\ref{98}) in (\ref{88}) for $\rho = 100r_{m}$, we get
\be \label{99}  d_{x_{im},100 r_{m}}(P_{jk},P_{im}) \leq C(n, C_{M}, C_{P}) \, \alpha(\bar{z},r_{k-4}), \quad \forall k \geq 1. \ee

This finishes the proof of (\ref{91}).\\
\\
Hence, we have shown that 
$\epsilon^{'}_{k}(f_{k}(z)) \leq C(n, C_{M}, C_{P}) \,  \alpha(\bar{z},r_{k-4})$ for every $k\geq 1$, that is
\be \label{136}  \epsilon^{'}_{k}(f_{k}(z))^{2} \leq C(n, C_{M}, C_{P}) \,  \alpha^{2}(\bar{z},r_{k-4}), \quad \forall k\geq 1\ee
Summing both sides of (\ref{136}) over $k \geq 0$, and using (\ref{a}) in Lemma \ref{l5} together with the fact that $\bar{z} \in M\cap B_{\frac{1}{10^{3}}}(0)$, we get
\be \label{100} \sum_{k=0}^{\infty} \epsilon^{'}_{k}(f_{k}(z))^{2} \leq 1 + C(n, C_{M}, C_{P}) \, \sum_{k=10}^{\infty} \alpha^{2}(\bar{z},r_{k-4}) \leq 1 + C(n, C_{M}, C_{P}) \, \epsilon^{2}_{0} \,\, := N.
\ee
Inequality (\ref{89}) is proved, and our theorem follows.

\end{proof}

\begin{remark} \label{cargen}
Notice that to prove Theorem \ref{MTT'}, the Carleson condition (\ref{103}) was needed to get the two inequalites
 
\be \label{a''again} \alpha(x,r) \leq C \, \epsilon_{0} \quad \forall \, x \in M \cap B_{\frac{1}{10^{3}}}(0) \quad \txt{and} \quad \forall \, 0 <  r \leq \frac{1}{10}\ee

and 

 \be \label{aagain} \sum_{j=1}^{\infty} \alpha^{2}(x, 10^{-j}) \leq C \, \epsilon^{2}_{0} , \quad \forall \, x \in M \cap B_{\frac{1}{10^{3}}}(0), \ee 

where $C = C(n, C_{M})$.
While (\ref{a''again}) was used all through this section, the only place (\ref{aagain}) was used was for inequality (\ref{100}). Now, notice that proving (\ref{100}) did not require that $ \displaystyle \sum_{j=1}^{\infty} \alpha^{2}(x, 10^{-j})$ be small, but just be finite (with the same upper bound for all points $\bar{z} \in M \cap B_{\frac{1}{10^{3}}}(0)$). Thus, Theorem \ref{MTT'} could be restated as follows: 
\end{remark}

\begin{theorem} \label{MTT'again}
Let $M \subset B_{1}(0)$ be an $n$-Ahlfors regular rectifiable set containing the origin, and let $\mu =$  \( \mathcal{H}^{n} \mres M\) be the Hausdorff measure restricted to $M$. Assume that $M$ satisfies the Poincar\'{e}-type inequality (\ref{eqp}). There exists $\epsilon_{0} = \epsilon_{0}(n, C_{M}, C_{P})>0$, such that if there exists a choice of unit normal $\nu$ to $M$ so that

\be \label{103again} \alpha(x,r)  < \epsilon_{0}  \quad \forall \, x \in M \cap B_{\frac{1}{10^{3}}}(0) \quad \txt{and} \quad \forall \, 0 <  r \leq \frac{1}{10},\ee
 then there exists a bijective map $g: \mathbb{R}^{n+1} \rightarrow \mathbb{R}^{n+1}$, and an $n$-dimensional plane $\Sigma_{0}$, with the following properties:
\ben g(z)= z \quad \txt{when} \,\,\, d(z, \Sigma_{0}) \geq 2, \een
and
\ben |g(z)-z| \leq C_{0} \epsilon_{0} \quad \txt{for} \,\,\, z \in \mathbb{R}^{n+1}, \een
where $C_{0} = C_{0}(n, C_{M}, C_{P})$,
\ben g(\Sigma_{0})\,\, \txt{is a} \,\,\, C_{0} \epsilon_{0} \txt{-Reifenberg flat set}, \een and 
\ben M \cap B_{\frac{1}{10^{3}}}(0) \subset g(\Sigma_{0}). \een

Moreover, if in addition to the assumptions above, there exists $N \in \mathbb{N}$ such that
\ben \sum_{j=1}^{\infty} \alpha^{2}(x, 10^{-j}) \leq N, \quad \forall \, x \in M \cap B_{\frac{1}{10^{3}}}(0), \een 
 then the map $g$ is $K$ bi-Lipschitz, where the bi-Lipschitz constant $K = K(n, C_{M}, C_{P}, N)$.

\end{theorem}

\section{Quasiconvexity of $M$} \label{QCM}

In this section, we show that the  Poincar\'{e}-type inequality (\ref{eqp}) that $M$ satisfies encodes some geometric information of $M$. More precisely, consider the metric measure space $(M, d_{0}, \mu)$ , where $M \subset B_{1}(0)$ is an $n$-Ahlfors regular rectifiable set in $\mathbb{R}^{n+1}$, $\mu =$  \( \mathcal{H}^{n} \mres M\) is the Hausdorff measure restricted to $M$., and $d_{0}$ is the restriction of the standard Euclidean distance in $\mathbb{R}^{n+1}$ to $M$ (which is obviously a metric on $M$). Our goal in this section is to show that if $M$ satisfies the  Poincar\'{e}-type inequality (\ref{eqp}), then $(M,d_{0},\mu)$ is quasiconvex.\\

\begin{definition}
A metric space $(X,d)$ is $\kappa$-quasiconvex if there exists a constant $\kappa \geq 1$ such that for any two points $x$ and $y$ in $X$, there exists a rectifiable curve $\gamma$ in $X$, joining $x$ and $y$, such that $\txt{length}(\gamma) \leq \kappa \, d(x,y)$.
\end{definition}

S. Keith proved a very nice theorem in his paper paper \cite{K1}, concerning the quasiconvexity of metric measure spaces supporting Poincar\'{e}-type inequalities. We are especially interested in a specific Poincar\'e-type inequality from \cite{K1}. To state his theorem with that Poincar\'e inequality, we first need to recall the notions of a doubling measure and a \emph{local Lipschitz constant function} on a metric measure space $(X,d,\nu)$.\\

\begin{definition}
Let $(X,d, \nu)$ be a metric measure space. We say that $\nu$ is a doubling measure if there is a constant $\kappa_{0} >0$ such that 
\ben \nu\left(B^{X}_{2r}(x)\right)\leq \kappa_{0} \, \nu\left(B^{X}_{r}(x)\right), \een where $x \in X$, $r>0$, and $B_{r}^{X}(x)$ denotes the metric ball in $X$, center $x$, and radius $r$.
\end{definition}

\begin{definition} Let $f$ be a Lipschitz function on a metric measure space $(X,d,\nu)$. The local Lipschitz constant function of $f$ is defined as follows
\be \label{lip} Lipf(x) = \lim_{r \to 0} \sup_{y \in B^{X}_{r}(x), \,y \neq x} \frac{|f(y) - f(x)|}{d(y,x)}, \,\,\,\,\, x \in X,\ee
where $B_{r}^{X}(x)$ denotes the metric ball in $X$, center $x$, and radius $r$.
\end{definition}

\textbf{Notation:} Let us note here that for any Lipschitz function $f$, $LIPf$ denotes the usual Lipschitz constant (see sentence below (\ref{lip1})), whereas $Lipf(.)$ stands for the local Lipschitz constant function defined above.\\

\begin{theorem} \label{qc} (see \cite{K1}, Lemma 9)
Let $(X,d,\nu)$ be a complete metric measure space, with $\nu$ a doubling measure. Let $\mathcal{B}$ be the collection of all balls in $X$, and assume that every ball in $X$ has positive and finite measure. Moreover, assume there is a constant $\kappa_{1} \geq 1$, such that for every Lipschitz function $f$ on $X$, and for every $B \in \mathcal{B}$, we have
\be \label{fct} \dashint_{B} \left| f(y) - f_{B}\right| \, d \nu(y) \leq \kappa_{1} \, \txt{diam}(B) \left(\,\dashint_{2B}(Lipf(y))^{2} \, d \nu(y) \right)^{\frac{1}{2}}, \ee
where $f_{B} := \dashint_{B} f \, d \nu$. Then $(X,d,\nu)$ is $\kappa$-quasiconvex, with $\kappa = \kappa(\kappa_{0}, \kappa_{1})$.
\end{theorem}

We want to use Theorem \ref{qc} to prove the main theorem of this section:

\begin{theorem} \label{mt}
Let $(M, d_{0}, \mu)$ be the metric measure space where $M \subset B_{1}(0)$ is $n$-Ahlfors regular rectifiable set in $\mathbb{R}^{n+1}$, $\mu =$  \( \mathcal{H}^{n} \mres M\) is the Hausdorff measure restricted to $M$, and $d_{0}$ is the restriction of the standard Euclidean distance in $\mathbb{R}^{n+1}$ to $M$. Suppose that $M$ satisfies the Poincar\'{e}-type inequality (\ref{eqp}). Then $(M, d_{0}, \mu)$ is $\kappa$-quasiconvex, with $\kappa = \kappa(n,C_{M}, C_{P})$.
\end{theorem}

The following lemma which is needed to prove the theorem above appears in \cite{KT} (p.379, Lemma 2.1). But for the sake of completion, we include the proof here.\\

\begin{lemma} \label{l} Let $M$ be an $n$-Ahlfors regular rectifiable subset of $\mathbb{R}^{n+1}$, and let $\mu =$  \( \mathcal{H}^{n} \mres M\) be the Hausdorff measure restricted to $M$. Let $x$ be a point in $M$ such that the approximate tangent plane $T_{x}M$ at $x$ exists. Consider a sequence $\{h_{i}\}_{i \in \mathbb{N}}$ of positive real numbers such that $ h_{i}\xrightarrow[i \to \infty]{} 0$, and for every $i \in \mathbb{N}$, let $M_{i} = \displaystyle\frac{M-x}{h_{i}}$. Then, for every $a \in T_{x}M$, there exists a sequence  $\{a_{i}\}_{i \in \mathbb{N}}$, with $a_{i} \in M_{i}$ for all $i \in \mathbb{N}$, such that $a_{i}\xrightarrow[i \to \infty]{} a$.
\end{lemma}

\begin{proof}
Let $x$, $\{h_{i}\}_{i \in \mathbb{N}}$, $\{M_{i}\}_{i \in \mathbb{N}}$, and $a$ be as stated above.
We first notice that it suffices to prove that $d(a,M_{i})\xrightarrow[i \to \infty]{} 0$. In fact, if the latter is satisfied, then for every $i \in \mathbb{N}$, let $a_{i} \in M_{i}$ such that $|a_{i} - a| \leq 2\,d(a,M_{i})$. Since $|a_{i} - a| \leq 2\,d(a,M_{i})\xrightarrow[i \to \infty]{} 0$, then, our sequence $\{a_{i}\}_{i \in \mathbb{N}}$, with $a_{i} \in M_{i}$ for all $i \in \mathbb{N}$, is such that $a_{i}\xrightarrow[i \to \infty]{} a$.\\
So, we prove that $d(a,M_{i})\xrightarrow[i \to \infty]{} 0$. We proceed by contradiction. Suppose that \\ $\displaystyle\lim_{i \to \infty}d(a,M_{i}) \neq 0$. Then, there exists an $\epsilon_{0} >0$, and a subsequence $\{M_{i_{k}} \}_{k \in \mathbb{N}}$ of $\{M_{i}\}_{i \in \mathbb{N}}$, such that $d(a,M_{i_{k}}) \geq \epsilon_{0}\,\,$ for every $k \in \mathbb{N}$. Thus, 

\be \label{200} B_{\frac{\epsilon_{0}}{2}}(a) \cap M_{i_{k}} = \emptyset, \,\,\,\,\, \forall k \in \mathbb{N}. \ee

Now, let $\varphi \in C_{c}^{\infty}(\mathbb{R}^{n+1})$ be a non-negative function on $\mathbb{R}^{n+1}$, such that $\varphi = 1$ on $B_{\frac{\epsilon_{0}}{4}}(a)$ and $\varphi = 0$ on $B^{c}_{\frac{\epsilon_{0}}{2}}(a)$. By the definition of the approximate tangent plane $T_{x}M$ at $x$, we know that 

\be \label{202} \lim_{k \to \infty} \frac{1}{h^{n}_{i_{k}}} \int_{M} \varphi\left(\frac{y-x}{h_{i_{k}}}\right) \, d\mathcal{H}^{n}(y) = \int_{T_{x}M} \varphi(y) \, d \mathcal{H}^{n}(y). \ee

Let us calculate the left hand side of (\ref{202}). Fix $k \in \mathbb{N}$. Then, for $y \in M$, we have 
$ \displaystyle \frac{y-x}{h_{i_{k}}} \in M_{i_{k}}$ which by (\ref{200}) implies that $ \displaystyle \frac{y-x}{h_{i_k}} \notin  B_{\frac{\epsilon_{0}}{2}}(a)$. However, we have chosen $\varphi$ such that spt$(\varphi)  \subset  B_{\frac{\epsilon_{0}}{2}}(a)$. Hence, we get
\be \label{203} \frac{1}{h^{n}_{i_{k}}} \int_{M} \varphi\left(\frac{y-x}{h_{i_{k}}}\right) \, d\mathcal{H}^{n}(y)= 0. \ee

Since (\ref{203}) holds for all $k \in \mathbb{N}$, then by plugging (\ref{203}) in (\ref{202}), we get

\be \label{204} \int_{T_{x}M} \varphi(y) \, d \mathcal{H}^{n}(y). \ee

Now, remembering that $\varphi = 1$ on $ B_{\frac{\epsilon_{0}}{4}}(a)$ and $\varphi \geq 0$, and using (\ref{204}), we get

\ben \omega_{n} \left(\frac{\epsilon_{0}}{4}\right)^{n} = \mathcal{H}^{n}( B_{\frac{\epsilon_{0}}{4}}(a) \cap T_{x}M)
=  \int_{ B_{\frac{\epsilon_{0}}{4}}(a) \cap T_{x}M} \varphi(y) \,d \mathcal{H}^{n}(y) 
= 0. \een
This is a contradiction, and thus the proof is done
\end{proof}

Now, let us turn our focus back to proving Theorem \ref{mt}. As we mentioned before, we want to apply Theorem \ref{qc} to prove Theorem \ref{mt}. In fact, we want to apply Theorem \ref{qc} to the metric measure space $(M, d_{0}, \mu)$. To do that, we show that the hypotheses of Theorem \ref{mt} imply those of Theorem \ref{qc}. In particular, we show that the Poincar\'{e}-type inequality (\ref{eqp}) from Theorem \ref{mt} implies the Poincar\'e-type inequality (\ref{fct}) from Theorem \ref{qc}.\\

 Let $(M, d_{0}, \mu)$ be the metric measure space where $M \subset B_{1}(0)$ is $n$-Ahlfors regular rectifiable set in $\mathbb{R}^{n+1}$, $\mu =$  \( \mathcal{H}^{n} \mres M\) is the Hausdorff measure restricted to $M$, and $d_{0}$ is the restriction of the standard Euclidean distance in $\mathbb{R}^{n+1}$ to $M$. Since $M$ is a closed and bounded subset of $\mathbb{R}^{n+1}$, then $M$ is complete. Now, let $\mathcal{B}$ be the collection of all metric balls in $(M, d_{0}, \mu)$, and take $B \in \mathcal{B}$. Let $x \in M$ be the center of $B$, and $r>0$ its radius. Denote such a ball by $B_{r}^{M}(x)$. It is trivial to see that
\be \label{ba} B_{r}^{M}(x) = B_{r}(x) \cap M, \ee
where $B_{r}(x)$ is the euclidean ball in $\mathbb{R}^{n+1}$ of center $x \in M$ and radius $r >0$. 
Notice that $\mu$ is doubling since it is Ahlfors regular, and for the same reason, every ball in $M$ has positive and finite measure. Hence, we are in the setting of Theorem \ref{qc}.\\

We want to show that the Poincar\'{e}-type inequality (\ref{eqp}) implies (\ref{fct}). To do that, we compare $|\nabla^{M}f|$ and $Lipf(.)$ when both of these functions are well defined. The following proposition gives us the relation between those two latter functions.

\begin{proposition} \label{l'}
Let $(M, d_{0}, \mu)$ be the metric measure space where $M \subset B_{1}(0)$ is $n$-Ahlfors regular rectifiable set in $\mathbb{R}^{n+1}$, $\mu =$  \( \mathcal{H}^{n} \mres M\) is the Hausdorff measure restricted to $M$, and $d_{0}$ is the restriction of the standard Euclidean distance in $\mathbb{R}^{n+1}$ to $M$. Let $f$ be a Lipschitz function on $M$. Then,
\ben |\nabla^{M} \bar{f}(x)| \leq Lipf(x) \,\,\,\,\, \mu\txt{-almost every}\,\, x \in M,\een
where $\bar{f}$ is a Lipschitz extension of $f$ to the whole space $\mathbb{R}^{n+1}$, with $f = \bar{f}$ on $M$, and $LIP\bar{f} \leq LIPf$.
\end{proposition}

\begin{proof}
Let $f$ be a Lipschitz function on $M$. Note that using the metric we have on $M$, we recall from (\ref{lip}) and (\ref{ba}) that

 \be \label{lip''} Lipf(x) = \lim_{r \to 0} \sup_{y \in B_{r}(x) \cap M, \,y \neq x} \frac{|f(y) - f(x)|}{|y-x|}, \,\,\,\,\, x \in M. \ee

By the well known Mcshane-Whitney extension lemma, $f$ extends to a Lipschitz function $\bar{f}$ defined on $\mathbb{R}^{n+1}$, with $ f = \bar{f}$ on $M$, and $LIP\bar{f} \leq LIPf$. Fix $x \in M$ such that the approximate tangent plane $T_{x}M$ exists. We prove that 
\be \label{205} |\nabla^{M} \bar{f}(x)| \leq  Lipf(x). \ee
Since $M$ is rectifiable, then, by Theorem \ref{atp}, $\mu$- a.e. point in $M$ admits an approximate tangent plane. Thus, by proving (\ref{205}), we would have proved the theorem. \\
\\
Let $ \tau(x)$ be a unit vector in $T_{x}M$. We claim that

\be \label{206} |<\nabla^{M} \bar{f}(x), \tau(x)>| \leq Lipf(x). \ee 

To see this, consider a sequence $\{h_{i} \}_{i \in \mathbb{N}}$ of positive numbers, such that $h_{i}\xrightarrow[i \to \infty]{} 0$. By Rademacher's theorem, we have 

\be \label{207} \lim_{i \to \infty} \frac{| \bar{f}(x+h_{i}\tau(x)) - \bar{f}(x) - h_{i} \, <\nabla^{M} \bar{f}(x), \tau(x)>|}{h_{i}} = 0. \ee

For simplicity, let us use the notation $\epsilon_{i}$ for the quantity inside the limit in the left hand side of (\ref{207}). Thus, we get
\be \label{208}  \lim_{i \to \infty} \epsilon_{i} = 0. \ee

Now, from the definition of $\epsilon_{i}$, we have

\ben \left | \frac{| \bar{f}(x+h_{i}\tau(x)) - \bar{f}(x)|}{h_{i}} -  |<\nabla^{M} \bar{f}(x), \tau(x)>|  \right | \leq \epsilon_{i}, \quad \forall i \in \mathbb{N}, \een

\be \label{209} |<\nabla^{M} \bar{f}(x), \tau(x)>| \leq  \frac{| \bar{f}(x+h_{i}\tau(x)) - \bar{f}(x)|}{h_{i}} +  \epsilon_{i},\quad \forall i \in \mathbb{N}.\ee

Let us now focus on the first summand of (\ref{209}). We want to show that \ben \limsup_{i \to \infty} \displaystyle  \frac{| \bar{f}(x+h_{i}\tau(x)) - \bar{f}(x)|}{h_{i}} \leq Lipf(x).\een The reason why this inequality is not straight forward is that for $i \in N$, the point $x + h_{i}\tau$ is not necessarily in $M$ (recall from (\ref{lip''}), $Lipf(x)$ only considers the points $y$ that are in $M$ and do not coincide with $x$). To remedy this, we need to move the points $x + h_{i}\tau, \,\, i \in \mathbb{N}$ just a little bit, to get a sequence of points $\{y_{i}\}_{i \in \mathbb{N}}$ that (just like the sequence $\{x + h_{i} \tau\}_{i \in \mathbb{N}}$) still approaches the point $x$ and does not coincide with it, but unlike the sequence $\{x + h_{i} \tau\}_{i \in \mathbb{N}}$, lives in $M$.\\

We proceed to constructing the sequence $\{y_{i}\}_{i \in \mathbb{N}}$. Since $\tau(x) \in T_{x}M$, then by Lemma \ref{l}, there exists a sequence $\{a_{i}\}_{i \in \mathbb{N}}$, with $a_{i} \in \displaystyle \frac{M-x}{h_{i}}$ for all $i \in \mathbb{N}$, such that  $a_{i}\xrightarrow[i \to \infty]{} \tau(x)$. Writing 

\be \label{lbl11}  a_{i} = \frac{y_{i}-x}{h_{i}} \,\,\,\,\, \forall i \in \mathbb{N}, \ee
we get a sequence  $\{y_{i}\}_{i \in \mathbb{N}}$, with $y_{i} \in M$ for all $i \in \mathbb{N}$, such that 

\be \label{lbl12} \lim_{i \to \infty} \left| \frac{y_{i}-x}{h_{i}} - \tau(x) \right| = 0,\ee
that is,
\be \label{210}  \lim_{i \to \infty} \left| \frac{y_{i}-x - h_{i}\tau(x)}{h_{i}} \right| = 0. \ee

Notice that from the definition of the $a_{i}$'s in (\ref{lbl11}), and recalling that $ \displaystyle\lim_{i \to \infty} a_{i} = \tau(x)$, $\tau(x)$ is a unit vector, and $\displaystyle\lim_{i \to \infty}h_{i} = 0$,  we can easily see that 

\be \label{211}  \lim_{i \to \infty} |y_{i} - x| =  \lim_{i \to \infty} h_{i}|a_{i}| = 0. \ee

Moreover, from (\ref{lbl12}) and the fact that $\tau(x)$ is a unit vector, we have 
\be \label{212}   \lim_{i \to \infty} \left|\frac{y_{i} - x}{h_{i}} \right| =  \lim_{i \to \infty} |a_{i}| = 1. \ee

Thus, by (\ref{212}), there exits $i_{0} \in \mathbb{N}$ such that for all $i \geq i_{0}$, we have 
$ \left| y_{i} - x \right| \geq \displaystyle \frac{h_{i}}{2}$, that is $ y_{i} \neq x$, for all $i \geq i_{0}.$
However, since all the limits and inequalities from (\ref{208}) till (\ref{212}) still hold when we restrict $i$ to $i \geq i_{0}$, then without loss of generality, we can assume that 

\be \label{221} y_{i} \neq x \quad \forall \, i \in \mathbb{N}. \ee

To sum up, $\{y_{i}\}_{i \in \mathbb{N}}$ is a sequence of points in $M$ that approaches the point $x \in M,$ and does not coincide with it.\\

Now, for $i \in \mathbb{N}$, we can write

\be \label{213}  \frac{| \bar{f}(x+h_{i}\tau(x)) - \bar{f}(x)|}{h_{i}} \leq  \frac{| \bar{f}(x+h_{i}\tau(x)) - \bar{f}(y_{i})|}{h_{i}} +  \frac{| \bar{f}(y_{i}) - \bar{f}(x)|}{h_{i}}. \ee

Rewriting the first term of the right hand side of (\ref{213}) and remembering that $\bar{f}$ is Lipschitz, we have

\bes \label{214}  \frac{| \bar{f}(x+h_{i}\tau(x)) - \bar{f}(y_{i})|}{h_{i}} &=&  \frac{| \bar{f}(x+h_{i}\tau(x)) - \bar{f}(y_{i})|}{|y_{i} - x - h_{i}\tau(x)|}\,\cdot \,  \frac{|y_{i} - x - h_{i}\tau(x)|}{h_{i}} \nonumber \\
&\leq& LIP\bar{f} \, \frac{|y_{i} - x - h_{i}\tau(x)|}{h_{i}}.\ees
(note that in case $y_{i} - x - h_{i}\tau = 0$, (\ref{214}) is satisfied trivially).\\

Also by rewriting the second term of the right hand side of (\ref{213}) (using(\ref{221})), and remembering that the points $y_{i}$ and $x$ are in $M$, and that $\bar{f} = f$ on $M$, we get

\bes \label{215}  \frac{| \bar{f}(y_{i}) - \bar{f}(x)|}{h_{i}} =  \frac{| \bar{f}(y_{i}) - \bar{f}(x)|}{|y_{i} - x|} \cdot \frac{|y_{i}-x|}{h_{i}} =  \frac{| f(y_{i}) - f(x)|}{|y_{i} - x|} \cdot \frac{|y_{i}-x|}{h_{i}}. \ees

Plugging (\ref{214}) and (\ref{215}) in (\ref{213}), we get 

\be \label{216}  \frac{| \bar{f}(x+h_{i}\tau(x)) - \bar{f}(x)|}{h_{i}} \leq  LIP\bar{f} \, \frac{|y_{i} - x - h_{i}\tau(x)|}{h_{i}} +  \frac{| f(y_{i}) - f(x)|}{|y_{i} - x|} \cdot \frac{|y_{i}-x|}{h_{i}}. \ee

Since (\ref{216}) holds for all $i \in \mathbb{N}$, then by taking the $\displaystyle \limsup_{i \to \infty}$ on both sides of (\ref{216}), we get using (\ref{210}) and (\ref{212}) that

\be \label{218}  \limsup_{i \to \infty} \frac{| \bar{f}(x+h_{i}\tau(x)) - \bar{f}(x)|}{h_{i}} \leq  \limsup_{i \to \infty}\frac{| f(y_{i}) - f(x)|}{|y_{i} - x|}. \ee

But, using (\ref{211}), (\ref{221}), and remembering that $y_{i} \in M$, it is easy to check that 

\be \label{219}  \limsup_{i \to \infty}\frac{| f(y_{i}) - f(x)|}{|y_{i} - x|} \leq Lipf(x). \ee

Thus, plugging (\ref{219}) back in (\ref{218}), we get 

\be \label{neww}  \limsup_{i \to \infty} \frac{| \bar{f}(x+h_{i}\tau(x)) - \bar{f}(x)|}{h_{i}} \leq Lipf(x). \ee 

Finally, taking $\displaystyle \limsup_{i \to \infty}$ on both sides of (\ref{209}), and using (\ref{neww}) and (\ref{208}), we get

\ben |<\nabla^{M} \bar{f}(x), \tau(x)>| \leq Lipf(x), \een
hence finishing the proof of (\ref{206}). (\ref{205}) follows directly from (\ref{206}) after plugging in  $\displaystyle \frac{\nabla^{M}\bar{f}(x)}{|\nabla^{M}\bar{f}(x)|}$ for $\tau(x)$ (the case when $\nabla^{M} \bar{f}(x) = 0$ is trivial). Proving (\ref{205}), as mentioned earlier, finishes the proof of the proposition
\end{proof}

\begin{corollary} \label{cor}
Let $(M, d_{0}, \mu)$ be the metric measure space where $M \subset B_{1}(0)$ is $n$-Ahlfors regular rectifiable set in $\mathbb{R}^{n+1}$, $\mu =$  \( \mathcal{H}^{n} \mres M\) is the Hausdorff measure restricted to $M$, and $d_{0}$ is the restriction of the standard Euclidean distance in $\mathbb{R}^{n+1}$ to $M$. Assume that $M$ satisfies the Poincar\'{e}-type inequality (\ref{eqp}). Let $f$ be a Lipschitz function on $M$. Then, for every $x \in M$, and radius $r >0$, we have 
\be \label{FOP2} \dashint_{B^{M}_{r}(x)} \left| f(y) -f_{B^{M}_{r}(x)} \right| \, d \mu(y) \leq C_{P} \, r \, \left(\dashint_{B^{M}_{2r}(x)}(Lipf(y))^{2} \, d \mu(y) \right)^{\frac{1}{2}}. \ee
\end{corollary}

\begin{proof}
Let $f$, $x$, and $r$ be as described above. Since $f$ is Lipschitz on $M$, we can extend it to a Lipschitz function $\bar{f}$ defined on $\mathbb{R}^{n+1}$, with $ f = \bar{f}$ on $M$, and $LIP\bar{f} \leq LIPf$. By construction, $\bar{f}$ is Lipschitz and thus locally Lipschitz on $\mathbb{R}^{n+1}$. Thus, we can apply the  Poincar\'{e}-type inequality (\ref{eqp}) to $\bar{f}$ at the point $x$ and radius $r$ to get 

\be \label{222} \dashint_{B_{r}(x)} \left| \bar{f}(y) - \bar{f}_{x,r} \right| \, d \mu(y) \leq C_{P} \, r \, \left(\,\dashint_{B_{2r}(x)}|\nabla^{M}\bar{f}(y)|^{2} \, d \mu(y) \right)^{\frac{1}{2}}. \ee

Using the fact that $\bar{f} = f$ on $M$ for the left hand side of (\ref{222}), and Proposition \ref{l'} for the right hand side of (\ref{222}), the latter becomes 

\be \label{cst}  \dashint_{B_{r}(x)} \left| f(y) -f_{x,r} \right| \, d \mu(y) \leq C_{P} \, r \, \left(\,\dashint_{B_{2r}(x)}(Lipf(y))^{2} \, d \mu(y) \right)^{\frac{1}{2}}. \ee

\noindent Hence, (\ref{FOP2}) follows directly from (\ref{cst}), (\ref{ba}), and the fact that $\mu =$  \( \mathcal{H}^{n} \mres M\)
\end{proof}

We are finally ready to put the pieces together and prove Theorem \ref{mt}: \\

\underline{\textbf{\textit{Proof of Theorem \ref{mt}:}}}
\begin{proof} 
We have already argued that $(M,d_{0},\mu)$ is a complete metric measure space, with $d_{0}$ being the restriction of the standard Euclidean distance in $\mathbb{R}^{n+1}$ to $M$,  and $\mu =$  \( \mathcal{H}^{n} \mres M\). We have also already shown that $\mu$ is a doubling measure with $\kappa_{0}= C(n, C_{M})$ and that the measure of every ball in $M$ is positive and finite. Moreover, by Corollary \ref{cor}, we have that $M$ satisfies the Poincar\'e-type inequality (\ref{fct}) with $\kappa_{1} = \displaystyle \frac{C_{P}}{2}$. Hence, by applying Theorem \ref{qc} to the metric measure space $(M, d_{0}, \mu)$, we get that  $(M, d_{0}, \mu)$ is $\kappa$-quasiconvex, with $\kappa = \kappa(n, C_{M}, C_{P})$, and the proof is done
\end{proof}

\begin{remark} \label{equiv}
As was kindly pointed out by the referee, it is very interesting here to study the connection between the Poincar\'e inequality (\ref{eqp}) used in this paper and other Poincar\'e-type inequalities found in literature that imply quasiconvexity (see for example \cite{Ch}, \cite{DJS}, \cite{K1} \cite{KM}). Apriori, these Poincar\'e-type inequalilties are different from eachother because the right hand side varies according to the notion of ``derivative" used on the metric space. However, it has been shown (see \cite{K1}, \cite{K2}, \cite{KM}) that if $(X, d, \mu)$ is a complete metric measure space with $\mu$ a doubling measure, and such that every ball has positive and finite measure, then all these Poincar\'e inequalities found in \cite{Ch}, \cite{DJS}, \cite{K1}, and \cite{KM} are equivalent. It turns out that this is also true for the Poincar\'e type inequality (\ref{eqp}). It is shown by the author in \cite{M} that for a Lipschitz function $f$ on $M$, $|\nabla^{M}\bar{f}|$ agrees $\mu$-almost everywhere with an upper gradient of $f$, where $\bar{f}$ is any extension of $f$ to a Lipschitz function on $\mathbb{R}^{n+1}$. This is the key point to showing that (\ref{eqp}) is indeed equivalent the other Poincar\'e-type inequalities mentioned above.\end{remark}

\begin{remark}
By Remark \ref{equiv} above, notice that Theorem \ref{MTT'} still holds if one replaces the Poincar\'e inequality (\ref{eqp}) with any other notion of the Poincar\'e inequality found in the references mentioned above. The reason why we prefer stating Theorem \ref{MTT'} with (\ref{eqp}) is that on a recitifiable set $M$, the tangential gradient is the obvious choice for a notion of derivative. Moreover, one can find several Sobolev and Poincar\'e-type inequalities on rectifiable sets that have the tangential gradient on the right hand side of the inequality (see \cite{Si}, \cite{H}, \cite{Se3}). In fact, as mentioned earlier in the introduction, Semmes has proved in \cite{Se3} that the Poincar\'{e}-type inequality (\ref{eqp}) is satisfied by chord-arc surfaces with small constant (CASSC).

\end{remark}

\begin{ack}
The author would like to thank T. Toro for her supervision, direction, and numerous insights into the subject of this project. The author would also like to thank Max Engelstein for the insightful discussions about \cite{DT1}. Last but not least, the author thanks the anonymous referee for his/her thorough reading of this manuscript, and his/her extremely helpful observations and comments.
\end{ack}

\bibliography{6161}{}
\bibliographystyle{amsalpha}

\end{document}